\documentclass[12pt]{amsart}
\usepackage{amssymb,latexsym,amscd}
\usepackage{fullpage}
\newcommand{\Mdef}[2]{\newcommand{#1}{\relax \ifmmode #2 \else $#2$\fi}}

\newcommand{\sm }{\wedge}

\newcommand{\tensor}{\otimes}

\newcommand{\map}{\mathrm{map}}
\newcommand{\Hom}{\mathrm{Hom}}

\newcommand{\Ext}{\mathrm{Ext}}
\Mdef{\bhom}{\mathbf{\hat{H}om}}

\Mdef{\Mod}{\mathrm{mod}}

\newtheorem{thm}{Theorem}[section]
\newtheorem{lemma}[thm]{Lemma}
\newtheorem{prop}[thm]{Proposition}
\newtheorem{cor}[thm]{Corollary}

\theoremstyle{definition}

\newtheorem{defn}[thm]{Definition}

\newtheorem{summary}[thm]{Summary}
\newtheorem{example}[thm]{Example}

\newtheorem{remark}[thm]{Remark}

\newtheorem{convention}[thm]{Convention}

\newcommand{\qqed}{\qed \\[1ex]}
\renewenvironment{proof}[1][\hspace*{-.8ex}]{\noindent {\bf Proof #1:\;}}{\qqed}

\Mdef{\PH} {\Phi^H}
\Mdef{\PK} {\Phi^K}
\Mdef{\PL} {\Phi^L}
\Mdef{\PT} {\Phi^{\T}}

\Mdef{\ef}{E{\cF}_+}
\Mdef{\etf}{\widetilde{E}{\cF}}
\Mdef{\eg}{E{G}_+}
\Mdef{\etg}{\tilde{E}{G}}

\Mdef{\infl}{\mathrm{inf}}
\Mdef{\defl}{\mathrm{def}}
\Mdef{\res}{\mathrm{res}}
\Mdef{\ind}{\mathrm{ind}}
\Mdef{\coind}{\mathrm{coind}}

\Mdef{\univ}{\mathcal{U}}

\Mdef{\Fp}{\mathbb{F}_p}
\Mdef{\Zpinfty}{\Z /p^{\infty}}
\Mdef{\Zpadic}{\Z_p^{\wedge}}

\newcommand{\bi}{\begin{itemize}}
\newcommand{\be}{\begin{enumerate}}
\newcommand{\bc}{\begin{center}}
\newcommand{\bd}{\begin{description}}
\newcommand{\ei}{\end{itemize}}
\newcommand{\ee}{\end{enumerate}}
\newcommand{\ec}{\end{center}}
\newcommand{\ed}{\end{description}}

\newcommand{\adjunction}[4]{
\diagram
#1:#2 \rrto<0.7ex> &&
#3  \llto<0.7ex> :#4 
\enddiagram}
\newcommand{\lra}{\longrightarrow}
\newcommand{\lla}{\longleftarrow}

\newcommand{\iso}{\cong}

\Mdef{\we}{\mathbf{we}}
\Mdef{\fib}{\mathbf{fib}}
\Mdef{\cof}{\mathbf{cof}}
\Mdef{\BI}{\mathcal{BI}}
\newcommand{\Qeq}{\simeq_Q}

\newcommand{\fibre}{\mathrm{fibre}}

\newcommand{\ilim}{\mathop{ \mathop{\mathrm{lim}} \limits_\leftarrow} \nolimits}

\newcommand{\holim}{\mathop{ \mathop{\mathrm {holim}} \limits_\leftarrow} \nolimits}

\Mdef{\A}{\mathbb{A}}
\Mdef{\B}{\mathbb{B}}
\Mdef{\C}{\mathbb{C}}
\Mdef{\D}{\mathbb{D}}
\Mdef{\E}{\mathbb{E}}
\Mdef{\T}{\mathbb{T}}
\Mdef{\F}{\mathbb{F}}
\Mdef{\G}{\mathbb{G}}
\Mdef{\I}{\mathbb{I}}
\Mdef{\N}{\mathbb{N}}
\Mdef{\Q}{\mathbb{Q}}
\Mdef{\R}{\mathbb{R}}
\Mdef{\bbS}{\mathbb{S}}
\Mdef{\Z}{\mathbb{Z}}

\Mdef{\bA}{\mathbb{A}}
\Mdef{\bB}{\mathbb{B}}
\Mdef{\bC}{\mathbb{C}}
\Mdef{\bD}{\mathbb{D}}
\Mdef{\bE}{\mathbb{E}}
\Mdef{\bF}{\mathbb{F}}
\Mdef{\bG}{\mathbb{G}}
\Mdef{\bH}{\mathbb{H}}
\Mdef{\bI}{\mathbb{I}}
\Mdef{\bJ}{\mathbb{J}}
\Mdef{\bK}{\mathbb{K}}
\Mdef{\bL}{\mathbb{L}}
\Mdef{\bM}{\mathbb{M}}
\Mdef{\bN}{\mathbb{N}}
\Mdef{\bO}{\mathbb{O}}
\Mdef{\bP}{\mathbb{P}}
\Mdef{\bQ}{\mathbb{Q}}
\Mdef{\bR}{\mathbb{R}}
\Mdef{\bS}{\mathbb{S}}
\Mdef{\bT}{\mathbb{T}}
\Mdef{\bU}{\mathbb{U}}
\Mdef{\bV}{\mathbb{V}}
\Mdef{\bW}{\mathbb{W}}
\Mdef{\bX}{\mathbb{X}}
\Mdef{\bY}{\mathbb{Y}}
\Mdef{\bZ}{\mathbb{Z}}

\Mdef{\cA}{\mathcal{A}}
\Mdef{\cB}{\mathcal{B}}
\Mdef{\cC}{\mathcal{C}}
\Mdef{\mcD}{\mathcal{D}} 
\Mdef{\cE}{\mathcal{E}}
\Mdef{\cF}{\mathcal{F}}
\Mdef{\cG}{\mathcal{G}}
\Mdef{\mcH}{\mathcal{H}} 
\Mdef{\cI}{\mathcal{I}}
\Mdef{\cJ}{\mathcal{J}}
\Mdef{\cK}{\mathcal{K}}
\Mdef{\mcL}{\mathcal{L}}

\Mdef{\cM}{\mathcal{M}}
\Mdef{\cN}{\mathcal{N}}
\Mdef{\cO}{\mathcal{O}}
\Mdef{\cP}{\mathcal{P}}
\Mdef{\cQ}{\mathcal{Q}}
\Mdef{\mcR}{\mathcal{R}}
\Mdef{\cS}{\mathcal{S}}
\Mdef{\cT}{\mathcal{T}}
\Mdef{\cU}{\mathcal{U}}
\Mdef{\cV}{\mathcal{V}}
\Mdef{\cW}{\mathcal{W}}
\Mdef{\cX}{\mathcal{X}}
\Mdef{\cY}{\mathcal{Y}}
\Mdef{\cZ}{\mathcal{Z}}

\Mdef{\Bt}{\tilde{B}}
\Mdef{\Ct}{\tilde{C}}
\Mdef{\Et}{\tilde{E}}
\Mdef{\Ht}{\tilde{H}}
\Mdef{\Kt}{\tilde{K}}
\Mdef{\Lt}{\tilde{L}}
\Mdef{\Mt}{\tilde{M}}
\Mdef{\Nt}{\tilde{N}}
\Mdef{\Pt}{\tilde{P}}

\Mdef{\tA}{\tilde{A}}
\Mdef{\tB}{\tilde{B}}
\Mdef{\tC}{\tilde{C}}
\Mdef{\tE}{\tilde{E}}
\Mdef{\tH}{\tilde{H}}
\Mdef{\tK}{\tilde{K}}
\Mdef{\tL}{\tilde{L}}
\Mdef{\tM}{\tilde{M}}
\Mdef{\tN}{\tilde{N}}
\Mdef{\tP}{\tilde{P}}

\Mdef{\ft}{\tilde{f}}
\Mdef{\xt}{\tilde{x}}
\Mdef{\yt}{\tilde{y}}

\Mdef{\Ab}{\overline{A}}
\Mdef{\Bb}{\overline{B}}
\Mdef{\Cb}{\overline{C}}
\Mdef{\Db}{\overline{D}}
\Mdef{\Eb}{\overline{E}}
\Mdef{\Fb}{\overline{F}}
\Mdef{\Gb}{\overline{G}}
\Mdef{\Hb}{\overline{H}}
\Mdef{\Ib}{\overline{I}}
\Mdef{\Jb}{\overline{J}}
\Mdef{\Kb}{\overline{K}}
\Mdef{\Lb}{\overline{L}}
\Mdef{\Mb}{\overline{M}}
\Mdef{\Nb}{\overline{N}}
\Mdef{\Ob}{\overline{O}}
\Mdef{\Pb}{\overline{P}}
\Mdef{\Qb}{\overline{Q}}
\Mdef{\Rb}{\overline{R}}
\Mdef{\Sb}{\overline{S}}
\Mdef{\Tb}{\overline{T}}
\Mdef{\Ub}{\overline{U}}
\Mdef{\Vb}{\overline{V}}
\Mdef{\Wb}{\overline{W}}
\Mdef{\Xb}{\overline{X}}
\Mdef{\Yb}{\overline{Y}}
\Mdef{\Zb}{\overline{Z}}

\Mdef{\db}{\overline{d}}
\Mdef{\hb}{\overline{h}}
\Mdef{\qb}{\overline{q}}
\Mdef{\rb}{\overline{r}}
\Mdef{\tb}{\overline{t}}
\Mdef{\ub}{\overline{u}}
\Mdef{\vb}{\overline{v}}

\Mdef{\hc}{\hat{c}}
\Mdef{\he}{\hat{e}}
\Mdef{\hf}{\hat{f}}
\Mdef{\hA}{\hat{A}}
\Mdef{\hH}{\hat{H}}
\Mdef{\hJ}{\hat{J}}
\Mdef{\hM}{\hat{M}}
\Mdef{\hP}{\hat{P}}
\Mdef{\hQ}{\hat{Q}}

\Mdef{\thetab}{\overline{\theta}}
\Mdef{\phib}{\overline{\phi}}

\Mdef{\uA}{\underline{A}}
\Mdef{\uB}{\underline{B}}
\Mdef{\uC}{\underline{C}}
\Mdef{\uD}{\underline{D}}

\Mdef{\bolda}{\mathbf{a}}
\Mdef{\boldb}{\mathbf{b}}
\Mdef{\boldD}{\mathbf{D}}

\Mdef{\fm}{\frak{m}}
\Mdef{\fp}{\frak{p}}

\Mdef{\eps}{\epsilon}

\input{xypic}

\newcommand{\Rt}{R_t}
\renewcommand{\Rp}{R'}
\newcommand{\kt}{k_t}
\newcommand{\ktop}{k_{top}}

\newcommand{\Rtopq}{R_{top}'}

\newcommand{\Rtopp}{R_{top}}
\newcommand{\RTOP}{R_{top}}
\newcommand{\Rtoppb}{\overline{R}_{top}}

\newcommand{\Ra}{R_a}

\newcommand{\Etop}{\cE_{top}}
\newcommand{\Etilde}{{\cE}'}
\newcommand{\Efunc}{E'}
\newcommand{\Ea}{\cE_a}
\newcommand{\cEa}{\cE_a}

\newcommand{\SGSS}{\mathcal{S}_{\Sigma}[G]}

\newcommand{\FHom}{\mathrm{F}}
\newcommand{\HomR}{\mathrm{Hom}_{R}}
\newcommand{\HomRp}{\mathrm{Hom}_{\Rp}}

\newcommand{\HomE}{\mathrm{Hom}_{\cE}}

\newcommand{\End}{\mbox{End}}

\newcommand{\FS}{F_{S_{G}}}

\newcommand{\tensorR}{\otimes_R}

\newcommand{\modE}{\mbox{mod-$\cE$}}

\newcommand{\modEtilde}{\mbox{mod-$\Etilde$}}
\newcommand{\modEtop}{\mbox{mod-$\Etop$}}

\newcommand{\AcellM}{\mbox{$A$-cell-$\cM$ }}

\newcommand{\kcellRmodp}{\mbox{$k$-cell-$R$-mod}_p}

\newcommand{\kcellHBGmodp}{\mbox{$\Q$-cell-$\HBG$-mod}_p}

\newcommand{\kcellRtmod}{\mbox{$\Q$-cell-$\Rt$-mod}}
\newcommand{\kcellRtopqmod}{\mbox{$k$-cell-$\Rtopq$-mod}}
\newcommand{\kcellRtopqGmod}{\mbox{$k$-cell-$\Rtopq(G)$-mod}}
\newcommand{\kcellRtopqHmod}{\mbox{$k$-cell-$\Rtopq(H)$-mod}}
\newcommand{\kcellRtoppmod}{\mbox{$k$-cell-$\Rtopp$-mod}}

\newcommand{\ktcellRtmodp}{\mbox{$k_t$-cell-$\Rt$-mod}_p}
\newcommand{\kpcellRmodp}{\mbox{$k'$-cell-$R$-mod}_p}

\newcommand{\kcellRmod}{\mbox{$k$-cell-$R$-mod}}
\newcommand{\kcellRpmod}{\mbox{$k$-cell-$\Rp$-mod}}
\newcommand{\kcellRmodi}{\mbox{$k$-cell-$R$-mod}_i}
\newcommand{\kcellHBGmodi}{\mbox{$\Q$-cell-$\HBG$-mod}_i}
\newcommand{\Rmod}{\mbox{$R$-mod}}
\newcommand{\Rtmod}{\mbox{$\Rt$-mod}}
\newcommand{\Ramod}{\mbox{$\Ra$-mod}}

\newcommand{\torsRmod}{\mbox{tors-$R$-mod}}

\newcommand{\torsRamod}{\mbox{tors-$\Ra$-mod}}

\newcommand{\Rpmod}{\mbox{$R'$-mod}}

\renewcommand{\bc}{\sigma}

\renewcommand{\BI}{\cB \cI}

\newcommand{\varrow}[1]{\hbox to #1{\rightarrowfill}}

\DeclareMathOperator{\Ho}{{Ho}}

\newcommand{\modu}{\mbox{-mod}}

\newcommand{\dg}{DG}

\renewcommand{\cM}{\bM}
\newcommand{\M}{\bM}
\renewcommand{\cN}{\bN}
\renewcommand{\N}{\bN}
\renewcommand{\mcD}{\bD}

\newcommand{\uk}{\underline{k}}
\newcommand{\cone}{\mathrm{Cone}}

\newcommand{\modcat}[1]{\mbox{$#1$-mod}}

\newcommand{\Gspec}{\mbox{$G$-spectra}}
\newcommand{\Hspec}{\mbox{$H$-spectra}}
\newcommand{\freeGspec}{\mbox{free-$G$-spectra}}
\newcommand{\freeHspec}{\mbox{free-$H$-spectra}}
\newcommand{\torsHBGmod}{\mbox{tors-$H^*(BG)$-mod}}
\newcommand{\torsHBHmod}{\mbox{tors-$H^*(BH)$-mod}}
\newcommand{\HBG}{H^*(BG)}
\renewcommand{\eg}{EG_+}
\newcommand{\piG}{\pi^G}
\newcommand{\bSG}{\bS [G]}

\begin{document}
\title{ An algebraic model for free rational $G$-spectra for connected
compact Lie groups $G$.} 
\author{J.~P.~C.~Greenlees}
\address{Department of Pure Mathematics, The Hicks Building, 
Sheffield S3 7RH. UK.}
\email{j.greenlees@sheffield.ac.uk}
\author{B.~Shipley}
\thanks{The first author was partially supported by  the EPSRC under grant EP/C52084X/1
and the second author by the National Science Foundation under Grant No. 0706877.}
\address{Department of Mathematics, Statistics and Computer Science, University of Illinois at
Chicago, 510 SEO m/c 249,
851 S. Morgan Street,
Chicago, IL, 60607-7045, USA}
\email{bshipley@math.uic.edu}

\date{}

\begin{abstract}
We show that the  category of free rational $G$-spectra for a
connected compact Lie group $G$ is Quillen equivalent to 
the category of torsion  differential graded modules over 
the polynomial ring $H^*(BG)$. The ingredients are
the enriched Morita equivalences  of \cite{ss2}, 
the functors of \cite{s-alg} making rational spectra 
algebraic and Koszul duality and thick subcategory arguments
based on the simplicity of the derived category of a polynomial ring.

\begin{center} 
\end{center}
\end{abstract}
\maketitle
\thanks{}
\tableofcontents

\section{Introduction.}
\label{sec:Intro}
For some time we have been working on the project of giving
a concrete algebraic model for various categories of rational 
equivariant spectra. There are a number of sources of
complexity: in the model theory, the structured spectra, the
equivariant topology, the categorical apparatus and the algebraic 
models. We have found it helpful, both in our thinking and in 
communicating our results, to  focus on the class of free spectra, 
essentially removing three sources of complication but still leaving
a result of some interest. 

The purpose of the present paper is to give a small, concrete
and calculable model for free rational $G$-spectra for a connected
compact Lie group $G$. The first attraction is that it is 
rather easy to describe both the homotopy category of free $G$-spectra
and also the algebraic model. The homotopy
category coincides with the category of rational cohomology 
theories on free $G$-spaces; better still, on free spaces an equivariant
cohomology theory is the same as one in the naive sense 
(i.e., a contravariant functor on the category of free $G$-spaces
 satisfying the Eilenberg-Steenrod axioms and the wedge axiom). 
As for the algebraic model, we consider a suitable model category 
structure on differential graded (DG) torsion modules over the 
polynomial ring $\HBG =H^*(BG;\Q)$.

\begin{thm}
\label{thm:culmination}
For any connected compact Lie group  $G$, there is a zig-zag of Quillen equivalences
$$\freeGspec /\Q \Qeq \torsHBGmod  $$
of model categories. In particular their derived  categories are 
equivalent
$$Ho(\freeGspec /\Q) \simeq D(\torsHBGmod ) $$
as triangulated categories. 
\end{thm}

The theorem has implications at several levels. First, it gives a
purely algebraic model for free rational $G$-spectra;  a property not 
enjoyed by all categories of spectra. Second, the algebraic model is
closely related to the standard methods of equivariant topology, and 
can be viewed as a vindication of Borel's insights. Third, the model
makes it routine to give calculations of maps between free
$G$-spectra, for example by using the Adams spectral sequence 
(Theorem \ref{thm:ASS}) below, and various structural 
questions (e.g., classification of thick and localizing subcategories)
 become accessible. Finally, the proof essentially proceeds by showing
 that certain categories are rigid, in the sense that in a suitable
 formal context, there is a unique  category whose objects
 modules with torsion homology over a commutative
 DGA $R$ with $H^*(R)$ polynomial. 

In Section \ref{sec:change} we also describe the counterparts of
functors relating free $G$-spectra and free $H$-spectra where $H$ is a
connected subgroup  of $G$. There is some interest to this, since if $i:H\lra G$
denotes the inclusion, the restriction functor $i^*$ from $G$-spectra
to $H$ spectra, has both a left adjoint $i_*$ (induction) and  a right
adjoint $i_!$  (coinduction). Similarly,  if $r:H^*(BG)\lra H^*(BH)$ is the
induced map in cohomology, the restriction $r^*$ from
$H^*(BH)$-modules $H^*(BG)$-modules has a left adjoint $r_*$  (extension
of scalars) and a right adjoint $r_!$ (coextension of scalars). In
fact the middle functor $i^*$ corresponds to the rightmost adjoint
$r_!$ of the functors we have mentioned. It follows that $i_*$ corresponds to
$r^*$,  but the right adjoint $i_!$ corresponds to a new functor; at
the  derived level $r_!$ is equivalent to a left adjoint $r_!'$, and
$r_!'$ has a right adjoint $r^!$. In the end, just as $i_!\simeq
\Sigma^{-c} i_*$, where $c=\dim (G/H)$, so too $r^!\simeq \Sigma^{-c} r^*$.

\begin{convention}
Certain conventions are in force throughout the paper. The most important 
is that {\em everything is rational}: henceforth all  spectra  and homology 
theories are rationalized without
comment.  For example, the category of rational $G$-spectra will now 
be denoted `$\Gspec$'.  The only exception to this rule is in the
beginning of Section~\ref{sec:AppB} where we introduce our model of 
(integral) free $G$-spectra.  We also use the standard conventions that
`\dg' abbreviates `differential graded' and that `subgroup' means 
`closed subgroup'. 
We focus on homological (lower) degrees, with differentials reducing degrees;
for clarity, cohomological (upper) degrees are called {\em codegrees} and 
are converted to degrees  by negation in the usual way.
Finally, we write $H^*(X)$ for the unreduced cohomology of a space $X$
with rational coefficients. 
\end{convention}

\section{Overview.}
\label{sec:overview}

\newcommand{\Homu}{\underline{\mbox{Hom}}}
\newcommand{\Etopp}{{\mathcal{E}}_{top}'}
\newcommand{\freeGspectra}{\mbox{free-$G$-spectra/$\Q$}} 
\newcommand{\QGmod}{\mbox{$\Q[G]$-mod}} 
\newcommand{\cellRtopmod}{\mbox{$\Q$-cell-$\Rtopq$-mod}} 
\newcommand{\cellRtoppmod}{\mbox{$\Q$-cell-$\Rtopp$-mod}} 
\newcommand{\cellRtmod}{\mbox{$\Q$-cell-$\Rt$-mod}} 
\newcommand{\cellHBGmod}{\mbox{$\Q$-cell-$H^*(BG)$-mod}}

\subsection{Strategic summary.}
In this section we describe our overall strategy. Our task is to obtain 
a Quillen equivalence between the category of rational free $G$-spectra and
the category of DG torsion $H^*(BG)$-modules. In joining these two 
categories we have 
three main boundaries to cross: first, we have to pass from the realm of 
apparently unstructured homotopy theory to a category of modules, secondly
we have to pass from a category of topological objects (spectra) to a
category of  algebraic objects (DG vector spaces), and finally we have 
to pass from modules over an arbitrary  ring to modules
over a commutative ring. These three steps could in principle be taken in 
any order, but we have found it convenient to begin by moving to modules
over a ring spectrum, then to modules over a commutative ring
spectrum and finally passing from modules over a commutative ring spectrum 
to modules over a commutative DGA. In the following paragraphs
we explain the strategy, and point to the sections of the paper where
the argument is given in detail. In Section~\ref{subsec:diagram} we then
give a condensed outline of the argument.  Various objects have counterparts in 
several of the worlds we pass through. As a notational cue we indicate
this with a subscript, so that $\RTOP$ is a topological object (a ring
spectrum), $\Rt$ is an algebraic object (a DGA, but very large and 
poorly understood), and $\Ra$ is an algebraic object (another DGA, 
small and concrete).

\subsection{Outline.}
In the homotopy category of free $G$-spectra, all spectra are constructed
from free cells $G_+$, and rational $G$-spectra from rationalized free 
cells $\Q [G]$. There are a number of models for free rational $G$-spectra, but
for definiteness we work with the usual model structure on orthogonal 
spectra \cite{mm}, 
and localize it so that isomorphisms of non-equivariant rational 
homotopy become weak equivalences (see Section \ref{sec:AppB} for a fuller 
discussion.) 

The first step is an application of Morita theory, as given 
in the context of spectra by Schwede and Shipley \cite{ss2}. Since all free
rational $G$-spectra are constructed from the small object $\Q [G]$, 
and since the category is suitably enriched over spectra, the category is 
equivalent to the category of module spectra over the (derived) endomorphism 
ring spectrum
$$\Etop=\Hom_{\Q [G]}(\Q [G], \Q [G])\simeq \Q [G]. $$
Note that we have already reached a category of non-equivariant spectra: 
the category of free rational $G$-spectra is equivalent to the category 
of $\Q [G]$-modules in spectra (Proposition \ref{prop.module.model}). 

Having reached a module category (over the non-commutative ring $\Q [G]$), 
the next step is to  move to a category of modules over a commutative ring. 
For this (see Section \ref{sec:Koszul}) we mimic the classical 
Koszul duality between modules over an exterior algebra (such as $H_*(G)$) 
and torsion modules over a polynomial ring (such as $H^*(BG)$).  The equivalence
takes a $\Q [G]$-module spectrum $X$ to $X\sm_{\Q [G]}\Q$, viewed
as a left module over the (derived) endomorphism ring spectrum 
$$\Rtopq = \Hom_{\Q [G]}(\Q, \Q ).$$
The quasi-inverse to this takes an $\Rtopq$-module $M$ to 
$\Hom_{\Rtopq} (\Q, M)$. In the first instance this is a module over
$\Etopp =\Hom_{\Rtopq}(\Q,\Q)$, but note that multiplication by
scalars defines the double centralizer map
$$\kappa: \Etop \simeq \Q[G]\lra \Hom_{\Rtopq}(\Q,\Q)=\Etopp. $$
We may therefore obtain an $\Etop$-module from an $\Etopp$-module by 
restriction of scalars.  We  use an Adams spectral sequence argument
to  show that $\kappa$ is a weak equivalence (Section~\ref{sec:ASS}),
and it follows that the categories of $\Etop$-modules and
$\Etopp$-modules are Quillen equivalent.   

Now the Morita theory from~\cite{ss2} shows that the category of 
$\Q$-cellular $\Rtopq$-modules (i.e., those modules built from $\Q$
using coproducts and cofibre sequences) is equivalent to the category of
$\Etopp$-modules (see Section \ref{sec:Koszultop}). The usual model 
of such cellular modules is obtained from a model for all modules 
by the the process of cellularization of model categories. Since we
make extensive use of this, we  provide an outline in 
Appendix~\ref{sec-cell}.

Unfortunately, $\Rtopq$ is not actually commutative. Recalling that in 
the definition we take a convenient cofibrant and fibrant replacement
$\Q [EG]$ of $\Q$, the ring operation on $\Rtopq =
\Hom_{\Q [G]}(\Q [EG], \Q[EG])$  is composition. 
However an argument of Cartan from the algebraic case shows
that it is quasi-isomorphic to the commutative ring 
spectrum $\Rtopp= \Hom_{\Q [G]}(\Q [EG], \Q)$, where the ring operation 
comes from that on $\Q$ and the diagonal of $EG$. 
Thus the category of free rational $G$-spectra is equivalent to the category 
of $\Q$-cellular modules over the commutative ring spectrum $\Rtopp$
(Proposition \ref{comm.ring}).

Finally, we are ready to move to algebra in Section~\ref{sec:Alg}. 
We have postponed this step as
long as possible because it loses control of all but the most formal properties.
Here we invoke the second author's equivalence between algebra 
spectra $A_{top}$ over the rational Eilenberg-MacLane spectrum $H\Q$ 
and DGAs $A_t$
over $\Q$. Under this equivalence, the corresponding module categories are
Quillen equivalent. Because we are working rationally, if $A_{top}$ is 
commutative,  the associated DGA is weakly equivalent to a commutative DGA 
by~\cite[1.2]{s-alg}. We apply this to the 
commutative ring spectrum $\Rtopp$ to obtain a DGA 
which is weakly equivalent to a commutative DGA, $\Rt$. 
Since the equivalence preserves homotopy we find $H^*(\Rt)=H^*(BG)$, 
and since this is a polynomial ring, we may construct a homology 
isomorphism $H^*(BG) \stackrel{\simeq}\lra \Rt$ by choosing cocycle 
representatives for 
the polynomial generators (Lemma~\ref{lemma.2.9}). 
This shows the category of $\Rtopp$-modules
is equivalent to the category of modules over $H^*(BG)$. We want to take
the cellularization of this equivalence, and it remains to show that the 
image of the $\Rtopp$-module $\Q$ is the familiar $H^*(BG)$-module of the 
same name: for this we simply observe that all modules $M$ with 
$H^*(M)=\Q$ are equivalent (Lemma~\ref{lemma.2.10}). It follows that
the category of free rational $G$-spectra is equivalent to the category 
of $\Q$-cellular $H^*(BG)$-modules 
(see Section \ref{subsec:rigidity} for a fuller discussion).  

We have obtained a purely algebraic model for free rational $G$-spectra.
It is then convenient to make it a little smaller and more concrete by 
showing that there is a model structure on DG torsion $H^*(BG)$-modules
which is equivalent to $\Q$-cellular $H^*(BG)$-modules.
It follows that
the category of free rational $G$-spectra is equivalent to the category 
of torsion $H^*(BG)$-modules as stated in Theorem \ref{thm:culmination} 
(see Section \ref{sec:CMtorsRmod} for a fuller discussion). 

\subsection{Diagrammatic summary.}\label{subsec:diagram}
Slightly simplifying the above argument, we may summarize it by saying we 
have a string of Quillen equivalences
\begin{multline*}
\freeGspectra \stackrel{1}\simeq 
\QGmod 
\stackrel{2}\simeq 
\cellRtopmod 
\stackrel{3}\simeq 
\cellRtoppmod 
\\
\stackrel{4}\simeq 
\cellRtmod \stackrel{5}\simeq 
\cellHBGmod \stackrel{6}\simeq 
\torsHBGmod
\end{multline*}
The first line takes place in topology, and the second in algebra.
Equivalence (1) is a standard Morita equivalence.
Equivalence (2) combines a Koszul Morita equivalence with 
a completeness statement and 
equivalence (3) uses Cartan's commutativity argument.
Equivalence (4) applies the second author's algebraicization theorem
to move from topology to algebra.
Equivalence (5) uses the fact that $H^*(BG)$ is a polynomial ring and
hence intrinsically formal together with a recognition theorem for
the module $\Q$.
Equivalence (6) applies some more elementary algebraic Quillen  equivalences. 

\subsection{Discussion.}

In this paper we have presented an argument that is as algebraic as possible
in the sense that after reducing to $\Q [G]$-modules, one can imagine making
a similar argument (Koszul, completeness, Cartan, rigidity) in many 
different contexts with an algebraic character.
It is intriguing that the present argument leaves the world of 
$G$-spectra so quickly. Generally this is bad strategy in equivariant
matters, but its effectiveness here may be partly due to the fact we are
dealing with free $G$-spectra.   We use a different outline in~\cite{tnq3} 
for torus-equivariant spectra which are not necessarily free. 

One drawback of  the present proof is that the equivalences have
poor monoidal properties.  To start with, it is not entirely clear how to give
all of the categories symmetric monoidal products: for example 
some ingenuity is required for torsion  $H^*(BG)$-modules, since there
are not enough flat modules.  

 On the other hand, one can give a more topological argument. This 
relies on special properties of equivariant spectra, but it does
have the advantage of having fewer steps and better monoidal
properties. We plan to implement this in a future paper covering
more general categories of spectra.

The restriction to connected groups is a significant simplification.
To start with, the algebraic model for disconnected groups is more  complicated.
From the case of finite groups (where the category of 
free rational $G$-spectra is equivalent to the category of graded
modules over the group ring \cite{Tate}), 
we see that we can no longer expect a single
indecomposable generator. Further complication is apparent for $O(2)$ (a split
extension of $SO(2)$ by the group $W$ of order 2), where
the component group acts on $H^*(BSO(2))$ \cite{o2q,Barnes1,Barnes2}.
One expects more generally, that if $G$ is the split extension of the 
identity component $G_1$ by the finite group $W$, 
the algebraic model consists of torsion modules over the twisted group ring 
$H^*(BG_1)[W]$. For non-split extensions like $Spin (2)$, the situation 
is similar for homotopy categories, but the extension is encoded in
the model category. The importance of the restriction to 
connected groups in the 
proof emerges in the convergence of the  Adams spectral sequence.

\section{The Morita equivalence for spectra}
\label{sec:AppB}

There are many possible models for equivariant $G$-spectra.
The two most developed highly structured model categories for 
equivariant $G$-spectra are the category of 
$G$-equivariant EKMM $S$-modules \cite{ekmm,em}  
and the category of orthogonal $G$-spectra \cite{mm}.  
Due to technicalities in the proof of Proposition~\ref{prop-free-model}, we
use orthogonal $G$-spectra here (see Remark~\ref{3.2} for a more
detailed discussion). 
Recall that in this section we are not working rationally unless
we explicitly say so.
We need a model for rational free $G$-spectra, and this 
involves further choices.  As we show below, one simple model
for free $G$-spectra is the category of module spectra over $\bS[G]$, 
the suspension spectrum of $G$ with a disjoint basepoint added;
analogously, a model for rational free $G$-spectra is the category of
module spectra over the spectrum $H\Q[G] = H\Q \sm G_+$.  
In this section we show that other, possibly more standard,  models
for free $G$-spectra and rational free $G$-spectra agree with these
models.

\subsection{Models for integral free $G$-spectra.}
Any model for the homotopy theory of free $G$-spectra will have a 
{\em compact (weak) generator}~\cite[2.1.2]{ss2}
given by the homotopy type of the suspension spectrum of $G_+$, 
and we write $\bS[G]$ for a representative of this spectrum.
Morita theory~\cite[3.9.3]{ss2} (see also~\cite[1.2]{ss-mon}) then shows that, under
certain hypotheses,  a model category associated to
the homotopy theory of free $G$-spectra is 
equivalent to the category of modules over the derived endomorphism ring 
spectrum of this generator.  In the following, we show
that for free $G$-spectra the derived endomorphism ring spectrum 
of $\bS[G]$ is again weakly equivalent to $\bS[G]$. 

To be specific, here we model free $G$-spectra by using the category
of orthogonal $G$-spectra and the underlying non-equivariant 
equivalences, that is, those equivalences detected by non-equivariant
homotopy $\pi_*(-)$.     
In more detail, this model structure is given by cellularizing the model category of orthogonal
$G$-spectra with respect to $\bS[G]$ using Hirschhorn's
machinery~\cite[5.1.1]{hh} (see also Appendix~\ref{sec-cell}). Because
we are working with free $G$-spectra, all universes from $G$-fixed to
complete  give equivalent models, but for definiteness we will use
a complete $G$-universe. 
Here and for the rest of this section $\bS[G]$ will denote the 
orthogonal $G$-spectrum taking each 
$G$-inner product space $V$ in the universe to $\bS[G](V) = S^V \sm G_+$ 
where $S^V$ is the one point compactification of $V$. 
Since $G$ is compact and $[\bS[G], -]_*$ detects the non-equivariant 
equivalences, $\bS[G]$ is a compact  generator for this 
model category (see also Appendix~\ref{prop-gen}).  We will show that
this category of free $G$-spectra satisfies the hypotheses
of~\cite[3.9.3]{ss2}, and identify the derived endomorphism ring of $\bS[G]$
as $\bS[G]$, to establish the following Morita 
equivalence.

\begin{prop}\label{prop-free-model}
The model category of free $G$-spectra described above is
Quillen equivalent to the category of (right) $\bS[G]$-modules in 
orthogonal spectra.   In turn, this is Quillen equivalent to modules
over the suspension spectrum of $G_{+}$ in any of the highly structured symmetric monoidal model categories of spectra.
\end{prop}

\begin{proof}
The category of orthogonal G-spectra is compatibly tensored, 
cotensored and enriched over orthogonal spectra~\cite[II.3.2, III.7.5]{mm}.  
Specifically, given two orthogonal $G$-spectra $X$, $Y$ the 
enrichment over orthogonal $G$-spectra is denoted $F_{S_G}(X,Y)$, and
the enrichment over orthogonal spectra is given by taking $G$-fixed
points.   
Since $\bS[G]$ is cofibrant and
a compact generator of free $G$-spectra, it follows from~\cite[1.2]{ss-mon} 
that the model category of free $G$-spectra is Quillen equivalent
to modules over the  endomorphism ring spectrum $\End_G(f\bS[G])=
\FS (f\bS[G], f\bS[G])^G$ where $f$ is any fibrant replacement
functor in orthogonal $G$-spectra.   Since $\bS[G]$ is a monoid
and fibrant monoids are fibrant as underlying spectra, we
take $f$ to be the fibrant replacement in the category of monoids
in orthogonal $G$-spectra.  

We make use of the equivalences of ring spectra
$$\bS[G] \simeq (\bS[G])^{op}\simeq \FS(\bS[G], \bS[G])$$ 
where the first equivalence is the inverse map. 
For the second equivalence, note that the change of groups equivalence 
holds in the strong sense that since $\bS[G] = S_G \sm G_+$,
for any $G$-space $X$ the function spectrum $F_{S_G}(S_G \sm G_+, X)$ 
is given by the levelwise based function space $\map (G_+, X)$ whose $G$-fixed
point space is  just $X$.

Restriction and extension of scalars over a weak equivalence of orthogonal
ring spectra induce a Quillen equivalence between the categories of modules 
by~\cite[7.2]{ss-mon}; see also~\cite{mmss}.
Applying this several times we see that once we show that
$\FS(\bS[G], \bS[G])$  and $\FS(f\bS[G], f\bS[G])$ are 
weakly equivalent as ring spectra,  the model categories of module 
spectra over $\End_G(f\bS[G])$ and module spectra over $\bS[G]$ are Quillen 
equivalent.

We first show that we have a zig-zag of weak equivalences
of module spectra
$$\FS (\bS [G],\bS [G])\stackrel{\simeq }\lra
\FS (\bS [G],f\bS [G])\stackrel{\simeq}\twoheadleftarrow 
\FS (f\bS [G],f\bS [G]).$$
The first of these is 
a weak equivalence since it is isomorphic to 
$ \bS[G] \lra  f\bS [G]$.
The second map is a trivial fibration since the enrichment over
orthogonal spectra satisfies the analogue of SM7 \cite[12.6]{mmss} and
in the first variable we have a trivial cofibration  and the second variable 
is a fibrant object.  Since the fibrant replacement functor 
here is in the category of monoids, we must use the fact   
that trivial cofibrations of monoids forget to trivial cofibrations
by~\cite{ss1}.  

This zig-zag of module level equivalences is actually a
quasi-equivalence~\cite[A.2.2]{ss2}. Namely, these maps make   
$\FS (\bS [G],f\bS [G])$ a bimodule over $\End(\bS[G])$ and
$\End(f\bS[G])$ and the maps are given by right and left multiplication
with the fibrant replacement map $\bS[G] \to f\bS[G]$ 
(thought of as an ``element" of
the bimodule.) 
A zig-zag of equivalences of ring spectra can then be created from
this quasi-equivalence. Construct $\mathcal P$ as the pull back in 
orthogonal spectra of the quasi-equivalence diagram.  
$$\begin{array}{ccc}
\mathcal P&\stackrel{\simeq}\lra& \FS (f\bS [G],f\bS [G])\\
\downarrow \simeq&&\downarrow \simeq\\
\FS (\bS [G],\bS [G])&\stackrel{\simeq}\lra& \FS (\bS [G],f\bS [G])
\end{array}$$
Then $\mathcal P$
has the unique structure as a orthogonal ring spectrum such that the
maps $\mathcal P \to \FS(\bS[G], \bS[G])$  and
$\mathcal P \to \FS(f\bS[G], f\bS[G])$ are homomorphisms of ring $G$-spectra.
Since the right displayed map in the quasi-equivalence is a trivial fibration,
the pull back left map is also.   Since the bottom map is also a weak equivalence, 
the top map out of the pull-back is a weak
equivalence as well by the two out of three property.

According to~\cite[Cor. 1.2]{ss-mon}, the Morita theorem from~\cite[3.9.3]{ss2}
can be restated to give an equivalence with modules over endomorphism ring
spectrum in  any of the highly structured symmetric monoidal model
categories of spectra.  
It is also easy to see that the comparison
functors between these various models preserve suspension spectra.  
So, since $\bS[G]$ is a suspension spectrum, the comparison theorems 
of~\cite{mmss} and~\cite{schwede} show that one can model free $G$-spectra
in any of these settings as the module spectra over the suspension spectrum
of $G_+$. 
\end{proof}

\begin{remark}\label{3.2}
To model free $G$-spectra directly using~\cite{ekmm},  one would again need to
identify the derived endomorphism ring
of the suspension spectrum of $G_{+}$.   Since all spectra are fibrant in~\cite{ekmm},
one would then consider the endomorphism ring of a cofibrant replacement.  
To mimic the proof above one would need this to be a cofibrant
replacement as a ring spectrum at one point and a cofibrant
replacement as a module spectrum at another.  This technical
complication persuaded us to use orthogonal spectra instead. 
\end{remark}

\subsection{Models for rational free $G$-spectra.}
We now need a model for {\em rational} free $G$-spectra. 
One can localize any of the models for free $G$-spectra with respect to 
rationalized non-equivariant equivalences,
given by $\pi_*(-) \tensor \bQ$.  
The statements in this subsection hold in any of the highly structure monoidal model
categories of spectra;  
to be specific though, we use orthogonal spectra.
Very similar proofs hold in other models. In particular, let $L_{\bQ}(\bS[G]\modu)$ 
denote the localized model structure given by~\cite[4.1.1]{hh}.  
As above, one can also model this using spectral algebra; we will show
that another model for rational free $G$-spectra is given by
$H\Q [G]$-module spectra where $H\Q [G]$ denotes the 
spectrum $ H\bQ \sm \bS[G] \iso  H\bQ \sm G_+$.

\begin{prop}\label{prop.module.model}
The localization of free $G$-spectra with respect to rational
equivalences, $L_{\bQ}(\bS[G]\modu)$, is Quillen equivalent to the category 
of $H\Q[G]$-module spectra.  Thus $H\Q [G]$-module spectra is a model of
rational free $G$-spectra.
\end{prop}

\begin{proof}
Consider the functor $H\bQ \sm (-)$ from $L_{\bQ}(\bS[G]\modu)$
to $H\Q[G]$-modules.  
Since $\bQ$ is flat over $\pi^S_*$, one can use the Tor spectral
sequence from~\cite[IV.4.1]{ekmm} 
to show that the derived homotopy of $H\bQ \sm X$
is isomorphic to the rationalized derived homotopy of $X$; that is, 
$\pi_*(f(H\bQ\sm X)) \iso \pi_*(fX) \otimes \bQ$ for any cofibrant 
$\bS[G]$-module spectrum $X$.  
Thus, $H\bQ \sm -$ preserves weak equivalences.  Since
cofibrations in $L_{\bQ}(\bS[G]\modu)$ agree with those in $\bS[G]\modu$,
$H\bQ \sm -$ also preserves cofibrations
and is thus a left Quillen functor.  Denote its right adjoint by
$U$.

To show that these functors induce a Quillen equivalence, let 
$X$ be cofibrant in $L_{\bQ}(\SGSS\modu)$
and $Y$ be fibrant in $H\Q [G]\modu$.  The map $H\bQ \sm X \to Y$
is a weak equivalence if and only if 
$\pi_*f(H\bQ \sm X) \iso \pi_*Y$.  
By the above calculation for the source, this is true if and only if 
$\pi_*(fX) \otimes \bQ \iso \pi_*(Y)\otimes \bQ$
since $Y$ is a $H\Q[G]$-module and its homotopy groups are rational.
But this holds if and only if
the map $X \to UY$ is a rational weak equivalence.
\end{proof}
Note, the convention that everything is rational is back in 
place after this section. In particular, $H\Q [G]$ is henceforth
denoted by $\bS [G]$; it was denoted $\Q [G]$ in Section \ref{sec:overview}.

\section{An introduction to Koszul dualities.}
\label{sec:Koszul}
In this section we describe a general (implicitly enriched) strategy for proving
a Koszul duality result, 
$$\modE \simeq \kcellRmod ,$$
and illustrate it  in the motivating algebraic example (explicitly enriched over chain complexes).
We then give a proof of the topological case (explicitly enriched over spectra) 
in Section \ref{sec:Koszultop}. 

To describe the general strategy, we start from $\modE$, an enriched model 
category of modules over a ring $\cE$.
We consider a {\em bifibrant} (that is, cofibrant and fibrant) 
$\cE$-module $k$ and its enriched endomorphism ring, $\Rp=\HomE (k,k)$ 
(we will later use the simpler notation $R$ for an equivalent {\em commutative}
ring). Evaluation then gives $k$ 
a left $\Rp$-module structure, and we obtain a functor 
$$\Efunc:\Rpmod \lra \modEtilde$$
with $\Etilde= \HomRp (k,k)$ and 
$$\Efunc(M)=\HomRp (k,M).$$
This takes values in $\modEtilde$ since $\HomRp (k,M)$ is a 
right $\Etilde$-module by composition.  
By construction $k$ is a generator of $\kcellRpmod$, when $k$ 
is small in the homotopy category~\cite[2.1.2]{ss2} as an $\Rp$-module, 
so Morita theory applies here to give a Quillen equivalence 
between $\kcellRpmod$ and $\modEtilde$.  This statement is 
proved as part of Theorem~\ref{thm-kcell-modE} below. 

We may then restrict along the double centralizer map 
$$ \kappa: \cE \lra \Hom_{\HomE(k,k)}(k,k)=\HomRp (k,k) = \Etilde$$
adjoint to the action map.  Composing then
gives a functor 
$$E:\Rpmod \lra \modE$$ where
$E=\kappa^* \Efunc$, so that  $E(M) = \HomRp (k,M)$, now thought
of as an $\cE$-module. 
Provided the map $\kappa$ is a weak equivalence
we combine this with the Morita equivalence between
$\kcellRpmod$ and $\modEtilde$ mentioned above
to prove in  Theorem~\ref{thm-kcell-modE} below that 
we have a Quillen equivalence $$\kcellRpmod \simeq \modE.$$

To motivate the terminology and display the obstacles in a 
familiar context, we describe the classical example in our terms.
A subscript $a$ is added here to specify that this is the algebraic
case. 
\begin{example}
\label{eg:alg}
The algebraic example works with the ambient category 
of differential graded $k$ modules (i.e., chain complexes) for a field
$k$ and 
$$\cE= \Ea=H_*(G),$$
an exterior algebra on a finite dimensional vector space in odd degrees,
with trivial differential.
Up to equivalence (see more detail in the next paragraph) we find 
$$R= \Ra \simeq H^*(BG),$$
a  symmetric algebra on a finite dimensional vector space in even 
codegrees. The conclusion is the statement
$$\mbox{mod-$H_*(G)$} \simeq \mbox{$k$-cell-$H^*(BG)$-mod}.$$
This is a Koszul duality, since 
 we will describe in Section \ref{sec:CMtorsRmod} below how 
$k$-cell-$H^*(BG)$-mod is a particular
model for the category of torsion modules over the polynomial ring 
$H^*(BG)$.

It may be helpful to be a bit more precise, to highlight some of the
technical issues we need to tackle. Choosing the projective model 
structure on the category 
of $\cEa$-modules, all objects are fibrant and therefore
a bifibrant version of  $k$ is given by a projective resolution $k_c$ 
(in this case we could be completely explicit). Then 
we have $\Ra'=\Hom_{\cEa}(k_c,k_c)$. 
We then argue that we may choose $k_c$ so that it has a 
cocommutative diagonal map and hence a Cartan commutativity
argument (see Proposition~\ref{comm.ring}) gives a ring
map to the commutative DGA  $\Hom_{\cEa}(k_c,k)$. 
It is formal that this map is a weak equivalence; 
thus $\Ra'$ is equivalent to $\Ra=\Hom_{\cEa}(k_c,k)$. 
Since its homology is the polynomial 
ring $H^*(BG)$, we may pick 
representative cycles in even codegrees to give a homology isomorphism 
$$H^*(BG) \lra \Hom_{\cEa}(k_c,k).$$
Note, this implicitly uses the intrinsic formality of a polynomial
ring on  even degree generators (Lemma~\ref{lemma.2.9}).  

Returning to the main argument,  $k_c$ is a left $\Ra'$-module by construction. 
To form $\cEa'$, we replace $k_c$ by a version bifibrant in $\Ra'$-modules;
again, one could be completely explicit. 
To see that the comparison map $\kappa$ is a homology isomorphism, 
we use the classical calculation $\Ext_{H^*(BG)}^{*,*}(k,k)\cong H_*(G)$, 
which could be viewed as a collapsed Eilenberg-Moore spectral sequence. 
\end{example}

\section{The topological Koszul duality.}
\label{sec:Koszultop}
In this section we establish the Quillen equivalence
$$\modE \simeq \kcellRmod $$
in the topological setting (enriched over spectra)
using the strategy of Section \ref{sec:Koszul}.

Before turning to the proof, we describe the topological example, following the 
pattern of Example \ref{eg:alg}.   
The topological example works best with the ambient category of EKMM spectra
\cite{ekmm} because all spectra are fibrant in this model, so 
we work with EKMM spectra throughout this section; 
by Proposition \ref{prop-free-model} 
free rational $G$-spectra are modeled by $\bS[G]$-module spectra in EKMM 
spectra.   We then use the 
comparison functors of~\cite{schwede} to pass to symmetric 
spectra on the way to shifting to algebra in Section~\ref{sec:Alg}.  
Note as well that these comparison functors take the symmetric spectrum
$\SGSS$ to the EKMM suspension spectrum of $G_+$, $\bS[G]$.
\begin{example}
\label{eg:top} 
We take 
$$\cE=\Etop=\bSG ,$$
 the suspension spectrum of $G_+$ from~\cite{ekmm}.
The cofibrant objects in $\Etop$-modules are those built from $\bSG$
(i.e., retracts of $\bSG$-cell objects) 
and weak equivalences are created in spectra ($\bS$-modules). 
The sphere spectrum $\bS$ is an $\bSG$-module by using the augmentation 
$\bSG \lra \bS [*] =\bS$.  
Since $\bS$ is not cofibrant as an $\bSG$-module, we consider a cofibrant 
replacement 
$$\bS_c = \bS [EG] = \bS \sm EG_+.$$  
We then take 
$\ktop=\bS_{cf} = \bS[EG]$ since it is cofibrant and fibrant ({\em bifibrant})
and weakly equivalent to $\bS$.

Now, using the internal enrichment of EKMM spectra we define 
$$\Rtopq =\FHom_{\Etop}(\ktop,\ktop)=\FHom_{\bSG}(\bS_{cf},\bS_{cf}).$$
It is central to our formality argument that this is equivalent to a 
commutative ring spectrum. 
\begin{prop}\label{comm.ring}
The ring spectrum $\Rtopq =\FHom_{\bS [G]}(\bS [EG],\bS [EG])$ (with 
ring operation given by composition) is equivalent to the
 commutative ring spectrum
$\Rtopp =\FHom_{\bS [G]}(\bS [EG],\bS )$ (with ring operation coming 
from the diagonal map of the space $EG$). 
\end{prop}

\begin{proof}
First, we define the map:
\begin{multline*}
\theta : \Rtopp =\FHom_{\bS [G]}(\bS [EG],\bS )
\stackrel{\sm_{\bS} \bS [EG]}\lra 
\FHom_{\bS [G]}(\bS [EG]\sm_{\bS} \bS [EG],\bS \sm_{\bS}\bS [EG] )\\
=\FHom_{\bS [G]}(\bS [EG \times EG],\bS [EG] )
\stackrel{\Delta^*}\lra 
\FHom_{\bS [G]}(\bS [EG],\bS [EG] )=\Rtopq.
\end{multline*}
To see this is a ring map, we use the following commutative diagram
$$\diagram
&\bS [EG] \dlto_{\Delta}\drto^{\Delta}&   \\
\bS [EG \times EG] \drto_{1\times \Delta} \xto[5,1]_{1\tensor (\alpha \cdot
\beta)}&&\bS [EG \times EG]
\dlto^{\Delta \times 1} \ddto^{1\tensor \beta}\\
&\bS [EG \times EG\times EG] \ddto^{1\tensor 1\tensor \beta}&\\
&&\bS [EG\times *] \dlto^{\Delta \times 1}\\
&\bS [EG \times EG\times *]\ddto^{1\tensor \alpha\tensor 1}&\\
&&\\
&\bS [EG \times *\times *]&
\enddiagram$$
Following down the left hand side we obtain $\theta (\alpha \cdot \beta)$, 
whereas the right hand side gives $\theta(\alpha)\circ\theta(\beta)$
as required. 
\end{proof}

\begin{remark}
Note also that we have an equivalence of ring spectra
$$\Rtopp =\FHom_{\bS [G]}(\bS [EG],\bS )\simeq \FHom_{\bS}(\bS [BG],\bS )=\Rtoppb$$
if one wants to minimise the appearance of equivariant objects.

Noting that $\bS[G]$ is analogous to  chains on $G$
and the final term to cochains on $BG$ (both with coefficients
in $\bS$), the equivalence $\Rtopp\simeq \Rtoppb$
 is an analogue of the Rothenberg-Steenrod theorem.
\end{remark}

Finally, we consider the map 
$$\kappa_{top} : \Etop=\bSG \stackrel{\simeq}\lra 
\FHom_{\Rtopq} (\bS_{cf},\bS_{cf})=\Etilde_{top}.$$
Using crude versions of the arguments constructing the
 Adams spectral sequence of Section \ref{sec:ASS}, 
we show that this is an equivalence in Lemma \ref{cor:kappatop}.
\end{example}

The next result holds in great generality, but here we only consider
the two relevant cases of enrichment over chain complexes and over spectra.  
Let $\cE, \Etilde, \Rp$ and $k$ be as described  
at the start of Section \ref{sec:Koszul}.

\begin{thm}\label{thm-kcell-modE}
In the algebraic and topological settings considered here,
if $k$ is a bifibrant and compact $\Rp$-module
and the double centralizer map $ \kappa: \cE \lra  \Etilde$
is a weak equivalence,
then there is a Quillen equivalence
$$\kcellRpmod \simeq \modE.$$
\end{thm}

In both the algebraic case (Example \ref{eg:alg}) and the spectral 
case (Example \ref{eg:top}) we have verified the hypotheses of this 
theorem.  Thus we have given a proof of the usual Koszul duality 
statement, and of the topological version we require.  Since
$\Rtopq$ and $\Rtopp$ are weakly equivalent by 
Proposition~\ref{comm.ring}, their module categories are 
Quillen equivalent by~\cite[4.3]{ss1}.

\begin{cor}\label{cor.5.5}
There are Quillen equivalences
$$\modEtop \simeq \kcellRtopqmod \simeq \kcellRtoppmod.\qqed$$
\end{cor}

\begin{proof}[of \ref{thm-kcell-modE}]
This Quillen equivalence follows in two steps.
The simplest step is that restriction and extension
of scalars across a weak equivalence induces a Quillen
equivalence between the categories of modules  
$$ \modE \simeq_Q \modEtilde.$$
This holds by~\cite[4.3]{ss1}. 

The other step is to establish the Quillen equivalence with right adjoint
$$\Efunc:\kcellRpmod \to \modEtilde.$$
For the model structure on $\kcellRpmod$, we consider the 
cellularization of 
$\Rpmod$ with respect to the object $k$, see Proposition~\ref{prop 5.5}.
Here the fibrations are again the
underlying fibrations in $\Rpmod$ but the weak equivalences are
the maps $g: M \to N$ such that  $g_*: \HomRp (k,fM) \lra \HomRp (k,fN)$ are 
weak equivalences (of $\bS$-modules or chain complexes) where $f$
is a fibrant replacement functor in $\Rpmod$.  The cofibrant objects
here are the $k$-cellular objects; those built from $k$ by disk-sphere pairs.  
In particular, since $k$ was
bifibrant in $\Rpmod$ it is also bifibrant in $\kcellRpmod$.

Since $k$ is compact and it detects weak equivalences
in $\kcellRpmod$, it is a compact (weak) generator by~\cite[2.2.1]{ss2}.   
Since $k$ is cofibrant and fibrant in $\kcellRpmod$ as well
as $\Rpmod$, its derived endomorphism ring is $\Etilde = \HomRp (k,k)$.
It follows by~\cite[1.2]{ss-mon} for spectra or by~\cite{qs1q} for
differential graded modules that $\Efunc$ is the
right adjoint in a Quillen equivalence.
\end{proof}

\section{The Adams spectral sequence.}
\label{sec:ASS}
The following Adams spectral sequence is both an effective
calculational tool, and a quantitative version of a double-centralizer
statement. It states that for free $G$-spectra, the stable equivariant 
homotopy groups $\piG_*(X)=[S^0,X]^G_*$ give a complete and effective
invariant. 

\begin{thm} 
\label{thm:ASS}
Suppose $G$ is a connected compact Lie group. 
For any free rational $G$-spectra $X$ and $Y$ there is a natural
Adams spectral sequence
$$\Ext_{\HBG}^{*,*}(\piG_*(X) , \piG_*(Y))\Rightarrow [X,Y]^G_*.$$
It is a finite spectral sequence concentrated in rows $0$ to $r$ 
and strongly convergent for all $X$ and $Y$. \qqed
\end{thm}

This is useful for calculation, but we only need the less explicit
version in Subsection \ref{subsec:geomASS} for the proof of our main 
theorem. 

\subsection{Standard Operating Procedure.}

We apply the usual method for constructing an Adams spectral sequence
based on a homology theory $H_*$ on a category $\C$ with values in an 
abelian category $\cA$ (in our case $\C$ is the category of free rational
$G$-spectra, $\cA$ is the category of torsion $H^*(BG)$-modules
and $H_*=\piG_*$). It may be helpful to 
summarize the process: to construct an Adams spectral 
sequence for calculating $[X,Y]$ we proceed as follows. 

\noindent
{\bf Step 0:} Take an injective resolution 
$$0 \lra H_*(Y) \lra I_0 \lra I_1 \lra I_2 \lra \cdots $$
in $\cA$.  

\noindent
{\bf Step 1:} Show that enough injectives $I$ of $\cA$ (including the $I_j$)
can be realized by  objects $\bI$ of $\C$ in the sense that $H_*(\bI)\cong I$.
 
\noindent
{\bf Step 2:} Show that the injective case of the spectral sequence is 
correct in that homology gives an isomorphism
$$[X,Y] \stackrel{\cong}\lra \Hom_{\cA}(H_*(X),H_*(Y))$$
if $Y$ is one of the injectives $\bI$ used in Step 1.

\noindent
{\bf Step 3:} Now construct the Adams tower 
$$\diagram 
\vdots &\\
.\dto &\\
Y_2 \rto \dto & \Sigma^{-2} \bI_2\\
Y_1 \rto \dto & \Sigma^{-1} \bI_1\\
Y_0 \rto      & \Sigma^0 \bI_0\\
\enddiagram$$
over $Y=Y_0$ from the resolution. This is a formality from Step 2. 
We work up the tower, at each stage defining $Y_{j+1}$ to be
the fibre of $Y_j \lra \Sigma^{-j}\bI_j$, and noting that $H_*(Y_{j+1})$ 
is the $(j+1)$st syzygy of $H_*(Y)$.

\noindent
{\bf Step 4:} Apply $[X, \cdot ]$ to the tower. By the injective case
(Step 2), we identify the $E_1$ term with the complex $\Hom_{\cA}^*(H_*(X),I_{\bullet})$
and the $E_2$ term with $\Ext_{\cA}^{*,*}(H_*(X),H_*(Y))$.

\noindent
{\bf Step 5a:}  If the injective resolution is infinite, the first
step of convergence is to show that $H_*(\holim_j Y_j)$ is calculated
using a Milnor exact sequence from the inverse system $\{ H_*(Y_j)\}_j$,
and hence that $H_*(\holim_j Y_j)=0$. If the injective resolution  is
finite, this is automatic.

\noindent
{\bf Step 5b:} Deduce convergence from Step 5a by showing $\holim_j Y_j\simeq *$.
In other words we must show that $H_*(\cdot )$ detects isomorphisms
in the sense that $H_*(Z)=0$ implies $Z \simeq *$. In general, 
one needs to require that $\C$ is a category of appropriately complete objects for 
this to be true. This establishes conditional convergence. If $\cA$ has finite
injective dimension, finite convergence is then immediate.

\subsection{Implementation.}

In our case, Step 0 is a well known piece of algebra: 
torsion modules over the polynomial ring $H^*(BG)$ admit
injective resolutions consisting of sums of copies of the basic
injective $I=H_*(BG)$. Step 1 now follows from a familiar calculation.

\begin{lemma} 
 \label{piGegisI}
The injectives are realized in the sense that
$$\piG_*(\eg )=\Sigma^d I, $$
where $d$ is the dimension of $G$.
\end{lemma}
\begin{proof}
We calculate
$$\piG_*(EG_+)\cong \pi_*(BG^{ad})\cong H_*(BG^{ad})\cong \Sigma^dH_*(BG)=\Sigma^dI,$$
where the third isomorphism used the fact that since $G$ is connected, 
the adjoint representation is orientable. 
\end{proof}

Steps 3 and 4 are formalities, 
and Step 5a is automatic since the category of modules over the polynomial
ring $H^*(BG)$ is of finite injective dimension. This leaves Steps 2 and 5b.

For Step 2 we prove the injective case of the Adams spectral sequence. 

\begin{lemma}
\label{lem:ASSintoinj}
Suppose $Y$ is any wedge of suspensions of $EG_+$.
For any free $G$-spectrum $X$, application of $\piG_*$ induces an 
isomorphism
$$[X,Y ]^G \stackrel{\cong}\lra \Hom_{\HBG}(\piG_*(X),\piG_*(Y)). $$
\end{lemma}

\begin{proof}
Since $\piG_*(Y)$ is injective, both sides are cohomology theories
in $X$, so it suffices to establish the special case $X=G_+$. Since 
$G_+$ is small and $\piG_*(G_+)$ is finitely generated, it suffices to 
deal with the special case $Y=EG_+$. 

In other words, we must show that
$$\piG_*: [G_+,EG_+ ]^G \stackrel{\cong}\lra \Hom_{\HBG}(\piG_*(G_+),\piG_*(EG_+)) $$
is an isomorphism. Now both sides consist of a single copy of $\Q$ in degree 0, 
so it suffices to show the map is non-trivial. This is clear since  a non-trivial 
$G$-map $f: G_+ \lra EG_+$ gives a map $f/G: S^0=(G_+)/G\lra (EG_+)/G=BG_+$ which 
is non-trivial in reduced $H_0$.
\end{proof}

Finally, for Step 5b we prove the universal Whitehead Theorem. This is 
where the connectedness of $G$ is used.

\begin{lemma}
\label{univWhitehead}
If $G$ is connected, then the functor $\piG_*$ detects isomorphisms in the sense that
if $f: Y \lra Z$ is a map of free $G$-spectra inducing an isomorphism
$f_* : \piG_*(Y ) \lra \piG_*(Z)$ then $f$ is an equivalence. 
\end{lemma}

\begin{proof}  Since $\piG_*$ is exact, it suffices to prove that if $\piG_*(X)=0 $
then $\pi_*(X)=\piG_*(G_+ \sm X)=0$.

For a general 
connected group we may argue as follows provided we use the fact that
$\Rtopp \simeq F(EG_+, \bS )$ is a commutative ring spectrum. Indeed, if
$\HBG=\Q [x_1, \ldots , x_r]$ we may view each generator 
$x_i$ as a map of $\Rtopp$-modules. As such we may form the Koszul complex 
$$K_{\Rtopp}(x_1,\ldots,x_i)=\fibre (\Rtopp \stackrel{x_1} \lra \Rtopp) \sm_{\Rtopp}
\ldots \sm_{\Rtopp}\fibre (\Rtopp \stackrel{x_i}\lra  \Rtopp), $$
and we find 
$$K_{\Rtopp}(x_1, \ldots , x_r)\sm EG_+ \simeq G_+ . $$

Now we argue by  induction on $i$ that 
$$\piG_*(K_{\Rtopp}(x_1, \ldots , x_i) \sm X)=0.$$
The case $i=0$ is the hypothesis and the case $i=r$ is the required conclusion.
However the fibre sequence
$$K_{\Rtopp}(x_1, \ldots ,x_{i+1})\lra K_{\Rtopp}(x_1, \ldots ,x_{i})
\stackrel{x_i} \lra K_{\Rtopp}(x_1, \ldots ,x_{i})$$
shows that the $i$th case implies the $(i+1)$st.
\end{proof}

\begin{remark}
(i) If $G$ acts freely on 
a product of spheres we may argue as in \cite{tnq1}; this is essentially 
the same as the present argument, but spheres and Euler classes realize
the homotopy generators and avoid the need to use a triangulated 
category of $\Rtopp$-modules. 

(ii) There are other ways to make this argument with less technology.
Note that since we are working rationally and $X$ is free $\piG_*(X)=H_*(X/G)$
and $\pi_*(X)=H_*(X)$.

Suppose first that $X$ is a 1-connected space. We may consider the Serre
Spectral Sequence for the fibre sequence $G \lra X \lra X/G$. Since $G$
is connected $X/G$ is simply connected and the spectral sequence 
$$H_*(X/G; H_*(G))\Rightarrow H_*(X)$$
is untwisted. By hypothesis it has zero $E^2$-term and hence $H_*(X)=0$.
Now for a general spectrum we may express $X$ as a filtered colimit of 
finite spectra. Each of these finite spectra is a suspension of a space, 
and since the spaces are free, the suspensions may be taken to be by 
trivial representations. Accordingly there is a Serre Spectral Sequence
for each of the finite spectra, and by passing to direct limits, there
is a similar spectral sequence for $X$ itself, and we may make the same
argument. 
\end{remark}

\subsection{A geometric counterpart of the Adams spectral sequence.}
\label{subsec:geomASS}

The completion result that we need is really a non-calculational 
version of part of the construction of the Adams spectral sequence. 
Despite its non-calculational nature, it is convenient to use special 
cases of the arguments we have used earlier in the section, which is 
why we have placed it here. 

\begin{prop}
\label{prop:nu}
The enriched functor $\bS [EG]\sm_{\bSG}( \cdot )$ gives a natural 
equivalence
$$
\nu_{X,Y}:F_{\bSG}(X,Y) \lra 
F_{\Rtopp} (\bS [EG]\sm_{\bSG} X,\bS [EG]\sm_{\bSG} Y )
$$
of spectra for all $\bSG$-modules $X$ and $Y$.
\end{prop}

\begin{proof}
First note that,  for each fixed $Y$,  the class of $X$ for 
which $\nu_{X,Y}$ is an equivalence is closed under completion of triangles
and arbitrary wedges.  Next note that $\nu_{X,Y}$ is an equivalence 
when $X=\bSG$ and $Y=\bS [EG]$ by the argument of 
Lemma \ref{lem:ASSintoinj}. Since $\bSG$ is small, it follows that
$\nu$ is an equivalence for $X=\bSG$ when $Y$ is an arbitrary 
wedge of suspensions of modules $\bS [EG]$. It follows that for these $Y$, the 
map $\nu$ is an equivalence for  arbitrary $X$.

Next, note that,  for each fixed $X$,  the class of $Y$ for 
which $\nu_{X,Y}$ is an equivalence is closed under completion of triangles.  
The construction of  Adams
resolutions shows that an arbitrary $\bSG$-module is built from 
wedges of suspensions of $\bS [EG]$ using 
finitely many triangles, so it follows that $\nu$ is an equivalence
for arbitrary $X$ and $Y$ as claimed. 
\end{proof}

Finally, note that if  we take $X=Y=\bSG$ we obtain the double
centralizer map $\kappa_{top}$, giving the required corollary. 

\begin{cor}
\label{cor:kappatop}
The map
$$\kappa_{top} :
\bSG=F_{\bSG}(\bSG, \bSG) \stackrel{\simeq}\lra 
F_{\Rtopp} (\bS,\bS)=\cE_{top}'$$
is an equivalence.\qqed
\end{cor}

\begin{remark}
The proof of the corollary alone can be considerably simplified, since
we may fix $X=\bSG$, and we may show that $Y=\bSG$
is built from $Y=\bS [EG]$ by giving an explicit Adams 
resolution of $G_+$ modelled on the dual of the Koszul complex
as in \cite[Section 12]{tnq1}. 
\end{remark}

\section{From spectra to chain complexes.}\label{sec:Alg}

In this section we show that the topological model
category $\kcellRtoppmod$ is Quillen equivalent to
the algebraic model category $\kcellRtmod$.  We first
(implicitly) replace $\Rtopp$ by the associated symmetric ring spectrum
using the lax symmetric monoidal comparison functors in \cite{schwede}
or \cite[1.2]{ss-mon}; we do not change the notation for $\Rtopp$ here to 
emphasize the simplicity of changing models of spectra. 
To move from $H\Q$-algebras to rational DGAs
we use the functor  $\Theta$ developed in~\cite{s-alg};
explicitly, in the notation of \cite{s-alg} we have 
$\Theta = Dc (\phi^* N) Z c$. 

\begin{prop} 
The DGA $\Theta \Rtopp$ is weakly equivalent to a commutative rational
differential graded algebra $\Rt$.
\end{prop} 

\begin{proof}  
This follows from~\cite[1.2]{s-alg}.
\end{proof}

By~\cite[2.15]{s-alg} and~\cite[4.3]{ss1}, we then have the following.

\begin{cor}\label{cor-rqe}
There are Quillen equivalences 
between differential graded $\Rt$-modules and
$\Rtopp$-module spectra.
\[  \Rtopp\modu \simeq_{Q} \Rt\modu . \]
\end{cor}

We now apply Proposition~\ref{prop-cell-qe} to
the zig-zag of Quillen equivalences between $\Rtopp$-module
spectra and DG $\Rt$-modules from Corollary~\ref{cor-rqe}.   
We want to find  the algebraic equivalent of the $k$-cellularization
of $\Rtopp\modu$.  
Define $\kt$ to be the image of $k$ under the various derived Quillen 
functors involved in the zig-zag of equivalences between $\Rtopp$ and $\Rt$
modules.  This will involve some extension or restriction functors
across weak equivalences and functors from~\cite{s-alg}. 

\begin{cor}
The $k$-cellularization of $\Rtopp$-module spectra is
Quillen equivalent to the $\kt$-cellularization of $\Rt$. 
\[ \mbox{$k$-cell-$\Rtopp\modu$} \simeq_{Q} \mbox{$\kt$-cell-$\Rt\modu$}\]
\end{cor}

\section{Models of the category of torsion $H^*(BG)$-modules.}
\label{sec:CMtorsRmod}

\newcommand{\Rtmodp}{\mbox{$\Rt$-mod}_p}
\newcommand{\Rtmodi}{\mbox{$\Rt$-mod}_i}
\newcommand{\kkcellRtmodp}{\mbox{$k$-cell-$\Rt$-mod}_p}
\newcommand{\Ramodp}{\mbox{$\Ra$-mod}_p}
\newcommand{\Ramodi}{\mbox{$\Ra$-mod}_i}
\newcommand{\kkcellRamodp}{\mbox{$k$-cell-$\Ra$-mod}_p}
\newcommand{\kkcellRamodi}{\mbox{$k$-cell-$\Ra$-mod}_i}
\newcommand{\Rmodp}{\mbox{$R$-mod}_p}
\newcommand{\Rmodi}{\mbox{$R$-mod}_i}
\newcommand{\torsRamodi}{\mbox{tors-$\Ra$-mod}_i}
\newcommand{\Gammam}{\Gamma_{\fm}}

In this section we consider the category of $R$-modules, where 
$R$ is a commutative DGA over $k$ with $H^*(R)=\HBG$ 
(for example $R=\HBG$ or $R=\Rt$ with $k=\Q$).
More precisely, we need to consider models for the category of $R$-modules
with torsion homology. In general we can obtain a model by cellularizing
a model of all $R$-modules with respect to the residue field $k$;
 if $R$ has zero differential there is an alternative model with underlying
category consisting of the DG torsion modules. We will show that the various
models are Quillen equivalent.

\subsection{An algebraic template.}
\label{subsec:algexample}

We begin with an overview of a simple and explicit example. We give detailed 
proofs in Subsection \ref{subsec:piequiv} and \ref{subsec:rigidity} 
below after adapting it to the general case.

Consider the commutative ring $R=k[c]$ with $c$ of degree $-2$ 
(this corresponds to the case when $G$ is the circle group).  
We may consider the abelian category $\Rmod$ of DG $R$-modules, 
and  the abelian subcategory $\torsRmod$ of DG $R$-modules
$M$ with $M[1/c]=0$. These two categories are related by inclusion
and its right adjoint $\Gamma_c$ where $\Gamma_c M$ is the
submodule of elements annihilated by some power of $c$ 
$$\adjunction{i}{\torsRmod}{\Rmod}{\Gamma_c}.$$
We will put model structures on the  two categories so that this
becomes a Quillen equivalence.  On the other hand, there is a
second model structure on $\Rmod$ that  arises by 
comparison with topology (i.e., free rational $\T$-spectra when 
$k =\Q$); the two structures on $\Rmod$ are Quillen equivalent.

We begin by describing the  two relevant model structures on $\Rmod$: 
a projective model and an injective model. Both of them have the same
weak equivalences, and these are detected by the basic 0-cell
$$\uk =\cone(c: \Sigma^{-2}R \lra R)$$
which is an exterior algebra over $R$ on a class in degree $-1$
whose differential is multiplication by $c$.
The basic 0-cell $\uk$ has homology  the 
residue field $k$, but (unlike $\HomR (k , \cdot)$)  
the functor $\Hom_{R} (\uk, \cdot)$ preserves homology isomorphisms
since $\uk$ is free as an $R$-module.
The basic 0-cell $\uk$ embeds in the contractible object $C\uk$, which 
is the  mapping cone of the identity map of $\uk$. 

The projective model structure, $\kcellRmodp$, is cofibrantly generated, taking
the set of generating cofibrations $I$ to consist of suspensions of
$\uk \lra C\uk$, and the
 set of generating acyclic cofibrations $J$ to consist of suspensions of
$0 \lra C\uk$. The weak equivalences are maps $X \lra Y$ with 
the property that $\HomR(\uk, X) \lra \HomR (\uk, Y)$ is a homology 
isomorphism. The fibrations are surjective maps, and the cofibrations
are injective maps which are retracts of relative $I$-cellular objects.
We have been explicit, but this model structure is an instance of 
a general construction: it is the 
$\uk$-cellularization  of the usual projective model on $R$-modules
as in Proposition~\ref{prop 5.5}. 

For the injective model structure, $\kcellRmodi$, the weak equivalences are again  
maps $X \lra Y$ with the property that $\HomR(\uk, X) \lra \HomR (\uk, Y)$ 
is a homology isomorphism. We will not give explicit descriptions
of either i-fibrations or i-cofibrations, but the i-cofibrations are in 
particular monomorphisms and the i-fibrations are in 
particular surjections with kernels which are injective $R$-modules. 
This model structure is the $\uk$-cellularization  of the usual
injective model on $R$-modules.

The injective and projective model structures are Quillen equivalent 
using the identity functors. Using,  Proposition~\ref{prop-cell-qe} 
this follows by cellularizing the usual 
equivalence between projective and injective models
on the category of all DG $R$-modules. 

For $\torsRmod$ we must do something slightly different, since
$\HomR (k, \cdot)$ does not preserve homology isomorphisms, 
and $\uk$ is not in $\torsRmod$. Instead, weak equivalences are homology 
isomorphisms, cofibrations are monomorphisms, and the fibrant objects
are those which are $c$-divisible. Coproducts preserve torsion modules, 
so coproducts in the category $\torsRmod$ are familiar. However, the
product in $\torsRmod$ of a set $M_i$ of torsion modules is
$\Gamma_c \prod_i M_i$. This model structure is given in
\cite[Appendix B]{s1q}, but it also follows from the more general 
discussion in  Subsection \ref{subsec:itorsequiv}. 

Now the $(i, \Gamma_c)$ adjoint pair gives a Quillen equivalence between 
the injective model structure on $\Rmod$ and the model structure on 
$\torsRmod$. The point is that i-fibrant $R$-modules are already torsion
modules.

\begin{summary}
We have elementary Quillen equivalences
$$\torsRmod \simeq_Q \kcellRmodi  \simeq_Q \kcellRmodp .$$
\end{summary}

\subsection{The Koszul model of the residue field.}

We now return to the category of $R$-modules where $R$ is a commutative
DGA over $k$ with $H^*(R)=H^*(BG)$.
To get good control over the model structure, it is useful to choose
a good model $\uk$ of the residue field $k$ so that  
$\HomR(\uk, \cdot)$ preserves homology isomorphisms.

We use a standard construction from commutative algebra. If $B$ 
is a graded commutative ring with elements $x_1, \ldots , x_r$ 
we may form the Koszul complex
$$K(x_1,\ldots , x_n)=(\Sigma^{|x_1|}B \stackrel{x_1}\lra B) \tensor_B
\cdots \tensor_B (\Sigma^{|x_n|}B \stackrel{x_r}\lra B),  $$
which is finitely generated and free as a $B$-module. 
If $R$ is a commutative DGA this needs to be adapted slightly;
to build the one-element Koszul object we suppose given a cycle $x\in R$
and then define
$$K(x):=\cone (\Sigma^{|x|}R \stackrel{x}\lra R). $$
Up to equivalence this depends
only on the cohomology class of $x$. In our case $R$ has polynomial 
cohomology, and we  pick cycle 
representatives $x_i$ for the polynomial generators, and 
 form 
$$K(x_1,\ldots, x_n)=K(x_1)\tensorR K(x_2)\tensorR \cdots \tensorR K(x_r).$$
We note that this may also be formed from $R$ in $r$ steps by taking fibres.
Indeed, if we write $K_i=K(x_1,\ldots, x_i)$, then we have $K_0=R$ 
and there is a cofibre sequence
$$\Sigma^{|x_i|}K_{i-1}\stackrel{x_i}\lra K_{i-1}\lra K_i$$
for $i=1,2,\ldots r$.

\begin{defn}
\label{defn:flatbasiccell}
The Koszul  model of $k$ is 
the DG-$R$-module $\uk=K(x_1,x_2,\ldots , x_r)$
for a chosen set of cocycle representatives of the polynomial generators.
\end{defn}

The good behaviour of $k$-cellularization is based on compactness.

\begin{lemma}\label{lem-compact}
The basic 0-cell $\uk$ is compact in the sense that 
$$\bigoplus_i \HomR (\uk, M_i) \lra \HomR (\uk, \bigoplus_i M_i)$$
is an isomorphism for any $R$-modules $M_i$. 
\end{lemma}

\begin{proof}
As an $R$-module $\uk$ is free on $2^r$ generators.
The value of a map $f: \uk \lra M$ is determined by its values
on these generators.
\end{proof}

The Koszul model is designed to preserve weak equivalences, and
we need to know that a weak equivalence of torsion modules is a 
homology isomorphism.

\begin{lemma}
\label{Flatcellspreserveacyclics}
The Koszul model $\uk$ of $k$ preserves homology isomorphisms in the
sense that $\HomR(\uk, \cdot)$ takes homology isomorphisms
to homology isomorphisms.
\end{lemma}

\begin{proof} The object $\uk$ is formed from $R$ in $r$ steps
by taking mapping fibres. The object $\HomR (\uk, M)$ 
is therefore formed from $M=\HomR (R,M)$ by finitely many operations of taking
mapping cofibres. If $M$ is acyclic, so too is $\HomR (\uk,M)$.
\end{proof}

\begin{lemma}
\label{lem:kacycisacyc}
If $X$ and $Y$ have torsion homology then a map $f:X\lra Y$ is 
a homology isomorphism if and only if $\HomR(\uk,X)\lra \HomR(\uk,Y)$
is a homology isomorphism.
\end{lemma}

\begin{proof}
Considering the mapping cone $M$ of $f$, it suffices to show that if $M$ 
has torsion 
homology then $H_*(M)=0 $ if and only if $H_*(\HomR(\uk, M))=0$. 
It follows from \ref{Flatcellspreserveacyclics} that if $M$ 
is acyclic then so is $\HomR(\uk,M)$.

Conversely, suppose that $\HomR (\uk, M)$ is acyclic and $H_*(M)$ is torsion. 
Let $K_i=K(x_1,\ldots, x_i)$, and argue by downward induction on $i$ that
$\HomR (K_i,M)$ is acyclic. The hypothesis states that this is true
for $i=r$ and the conclusion is the statement that it is true for $i=0$.

Suppose then that $\HomR(K_i,M)$ is acyclic. Since $K_i $ is the fibre
of $x_i : K_{i-1} \lra K_{i-1}$, we conclude that 
$x_i$ is an isomorphism of $H_*(\HomR (K_{i-1},M))$. However, since 
$H_*(M)$ is torsion, so is  $H_*(\HomR (K_{i-1},M))$, and hence in particular
it is $x_i$-power torsion. An $\HBG$-module  $H$ for which $x_i:H\lra H$ 
is an isomorphism and $H[1/x_i]=0$ is zero. 
Hence $\HomR (K_{i-1},M)$ is acyclic as required. 
\end{proof}

\subsection{Models of cellular $R$-modules.}
\label{subsec:piequiv}

We now formally introduce the algebraic model structures we use. 
If $R$ is a DGA, we may form the projective model structure $\Rmodp$
(see~\cite[Section 7]{nj} or \cite{ss1}).
The weak equivalences are the homology isomorphisms. It is cofibrantly 
generated by using algebraic disk-sphere pairs. More precisely, we take
$S^{n-1}=S^{n-1}(R)$ to be the $(n-1)$st suspension of $R$, and 
$D^{n}(R)$ to be the mapping cone of its identity map. 
The set $I$ of generating cofibrations
consists of the maps $S^{n-1}\lra D^{n}$ for $n\in \Z$ and the set $J$ of
generating acyclic cofibrations consists of the maps $0\lra D^{n}$ for 
$n\in \Z$. The fibrations in this model are the surjective maps, and
the cofibrations are retracts of relative cell complexes. 

The proof may be obtained by adapting \cite[Start of 2.3]{hovey-model}. 
More precisely, 
$D^n$ is still right adjoint to taking the degree $n$ part and
$S^{n-1}$ is still right adjoint to taking degree $n-1$ cycles, so that
the proof that it is a model structure (i.e., until the end of the proof
of \cite[2.3.5]{hovey-model}) is unchanged. 
The analogue of \cite[2.3.6]{hovey-model}
states that any 
cofibrant object is projective if the differential is forgotten, and
that any DG-module $M$ admitting an inductive filtration 
$$0=F_0M \subseteq F_1M\subseteq F_2M\subseteq \cdots \subseteq \bigcup_n
F_nM=M$$
with subquotients consisting of sums of projective modules (i.e., a 
{\em cell $R$-module} in the sense of \cite{KM}) is cofibrant. The proof
of projectivity is unchanged, and the cofibrancy of cell modules is 
just as for CW-complexes (i.e., via the homotopy extension and lifting property
(HELP)
as in \cite[2.2]{KM}). Arbitrary cofibrant objects are retracts of cell 
$R$-modules. Cofibrations $i:A\lra B$ are retracts of relative cell 
modules.

The model category $\kcellRmodp$ is obtained from $\Rmodp$ by cellularizing
with respect to $k$ in the sense of \cite{hh}. The Koszul model
$\uk$ gives a convenient cofibrant model for $k$, and since all objects
are fibrant, the weak equivalences are the maps $p:X\lra Y$ for 
which $\HomR (\uk , p)$ is a homology isomorphism, the fibrations are the 
surjective maps as before. The model structure is cofibrantly generated, and
the sets $I$ and $J$ can be formed from $\uk =S^0(\uk)$ in the same way
that $I$ and $J$ were formed from $R=S^0(R)$ above. 

Similarly, in the special case that $R$ has zero differential, 
which is the only case we will need, 
the injective model $\Rmodi$ is formed as in \cite[2.3.13]{hovey-model}.
In fact Hovey only deals with rings and chain complexes, but it applies 
to a graded algebra $R$ with zero differential and DG modules over it.
The weak equivalences are the $H_*$-isomorphisms, and the 
cofibrations are the injective maps. The fibrations are the surjections 
with fibrant kernel.
We will not identify the fibrant objects explicitly, but they are in 
particular injective as modules. For the converse, we say that
 $M$ is {\em cocellular} if there is a filtration 
$$M=\ilim_n F^nM \lra \cdots \lra F^2M \lra F^1M \lra F^0M=0$$
where the kernels $K^n=\ker(F^nM \lra F^{n-1}M)$ are injective modules
with zero differential. A cocellular $R$-module is fibrant. 
The two places where it is necessary to adapt the
proof are the proof of \cite[2.3.17]{hovey-model} where bounded above modules
are replaced by cocellular modules (and the proof with elements replaced
by applications of the HELP), and in \cite[2.3.22]{hovey-model}. The zero 
differential
allows us to treat any $R$-module as a DG-module with zero differential.

We can form $\kcellRmodi$ by cellularizing
with respect to $k$ in the sense of \cite{hh}. In this case
$k$ is itself cofibrant, so the standard description of the 
weak equivalences is that they are maps $p:X\lra Y$ for which 
$\HomR (k , p')$ is a homology isomorphism, where $p'$ is a fibrant 
approximation of $p$; this is a little inconvenient, so we use the 
equivalent condition that $\HomR (\uk , p)$ is a homology isomorphism.
The fibrations are the same as for $\Rmodi$, and the cofibrations are
as required by the lifting property.

It is immediate from the above description of the model structures that we
have a Quillen equivalence 
$$\Rmodp\stackrel{\simeq}\lra\Rmodi$$
using the identity maps. By cellularizing this, see 
Proposition~\ref{prop-cell-qe}, we obtain a Quillen equivalence
$$\kcellRmodp\stackrel{\simeq}\lra\kcellRmodi.$$

\subsection{A model structure on torsion modules.}
\label{subsec:itorsequiv}

Finally, we consider a polynomial ring $\Ra$ on even degree generators
$x_1,\ldots , x_r$. A torsion module
is one for which every element is annihilated by a power of the 
augmentation ideal $\fm$ (or equivalently, it is annihilated by some power
of each element of  $\fm$).  There is an adjunction
$$\adjunction{i}{\torsRmod}{\Rmod}{\Gammam}, $$
where $\Gammam M$ is the submodule of elements annihilated by 
some power of the augmentation ideal $\fm$.

\begin{lemma}
The category $\torsRamod$ of DG torsion 
modules admits an injective model structure with
weak equivalences the  homology isomorphisms and cofibrations which 
are the injective maps. The i-fibrations are the surjective maps with 
i-fibrant kernel. 
\end{lemma}

\begin{proof}
One option is to observe that Hovey's proof of \cite[2.3.13]{hovey-model} 
applies to torsion modules. This is exactly as in the previous section 
except that to construct products and inverse limits one forms them in
the category of all $\Ra$-modules and then  applies the right adjoint 
$\Gammam$. Note that this shows in particular that the category of all 
torsion $\Ra$-modules does have enough injectives. 

Another option is to use the method of \cite[Appendix B]{s1q}. For the latter, 
we  need only show that the four steps described there 
can be completed. In fact we can use precisely the same argument, which reduces 
most of the verifications to properties of $\Q$-modules using 
right adjoints to the forgetful functor to vector spaces, together with the 
fact that $H^*(BG)$ is of finite injective dimension.
 The right adjoint is given by coinduction,  and the fact
that enough injectives are obtained in this way follows 
since any torsion module over a polynomial ring embeds in a
sum of copies of the injective hull of the residue field.
\end{proof}

\begin{prop}
The $(i,\Gammam)$-adjunction induces a Quillen equivalence
$$\torsRamodi \stackrel{\simeq}\lra \kkcellRamodi.$$
\end{prop}

\begin{proof}
First observe that since the cofibrations are injective maps in both 
cases and $\Gammam $ is right adjoint to inclusion, 
$\Gammam$ takes fibrant objects to fibrant objects.

Next we see that the adjunction is a Quillen pair. The fibrations in 
$\kkcellRamodi$ are the surjective maps $p: M \lra N$ with fibrant kernel $K$. 
Fibrant objects are in particular injective, and hence we obtain a short
exact sequence 
$$0 \lra \Gammam K \lra \Gammam M \lra \Gammam N \lra 0. $$
As observed above, since $K$ is fibrant, so is $\Gammam K$.
Thus $\Gammam$ takes fibrations to fibrations. 

Next, if $p$ is also a weak equivalence then $K$ is weakly contractible. 
Indeed,  from the exact sequence
$$0 \lra K\lra X\lra Y\lra 0$$
we see that $\HomR(\uk,K)$ is acyclic. 
However, 
$$\HomR (\uk, K)\simeq \HomR (k,K)=\Gammam K$$
 and hence $H_*(\Gammam K)=0$ as required. 

Finally, we check that the Quillen pair is a Quillen equivalence. 
For this suppose  $M$ is a torsion module and $N$ is fibrant. 
A map $p:M \lra  N$ is a weak equivalence if and only if the map
$$\HomR (\uk,M) \lra  \HomR (\uk,N)$$
is a homology isomorphism. Now $N$ and $\Gammam N$ are injective, 
so the diagram
$$\begin{array}{ccc}
\HomR (\uk , N)&\stackrel{\simeq}\lla& \HomR (k,N)\\
\uparrow &&\uparrow =\\
\HomR (\uk , \Gammam N)&\stackrel{\simeq}\lla& \HomR (k,\Gammam N)
\end{array}$$
allows us to deduce this is equivalent to the map
$\HomR (\uk, M)\lra \HomR (\uk, \Gammam N)$
being a homology isomorphism, and finally, 
 by Lemma \ref{lem:kacycisacyc}, this is equivalent to requiring  that 
$$M \lra \Gammam N$$
is a homology isomorphism as required.
\end{proof}

\subsection{Rigidity of the category of torsion modules.}
\label{subsec:rigidity}
In this subsection we give the algebraic part of the string of equivalences, 
following the pattern in Subsection \ref{subsec:algexample}.

The outcome of the topological part of the argument is the category
$\ktcellRtmodp$. 
The only thing we know about $\Rt$ is that it is a 
commutative DGA over $\Q$ with $H^*(\Rt)=H^*(BG)$, and the only thing we 
know about the object we have used to cellularize it is that its homology
agrees with that of a free cell.

 However this data is rigid in the following sense. 
\begin{prop}
If $R$ is a commutative DGA over $k$ with cohomology $\Ra$ polynomial on even 
degree generators, and if we cellularize with respect to any compact object
$k'$ with homology $k$, there are Quillen equivalences
$$\kpcellRmodp \simeq \kkcellRamodp\simeq \kkcellRamodi \simeq \torsRamod .$$
\end{prop}

\begin{proof}
We have described the model structures on all four categories.

For the first equivalence we have a well-known observation. 
\begin{lemma}\label{lemma.2.9}
A polynomial ring $\Ra$ is intrinsically formal amongst commutative
DGAs over $k$ in the sense that there is a homology isomorphism
$\Ra \lra R$.
\end{lemma}

\begin{proof}
We choose representative cycles for the polynomial generators. 
Since $R$ is commutative we may use them 
to define a map $\Ra \lra R$
of DGAs by taking each polynomial generator to the representative cycle.
 By construction it is a homology isomorphism. 
\end{proof}

Extension and
restriction of scalars therefore define a Quillen equivalence
$$\Rmodp\simeq \Ramodp. $$

We next need to cellularize this equivalence. On the left we use an 
object with homology equal to $k$, and the corresponding object on the 
right has the same property. The second ingredient in rigidity is that
 this characterizes the objects up to equivalence.

\begin{lemma}\label{lemma.2.10}
If $M$ is any $R$-module with $H_*(M) \cong k$ there is an 
equivalence $M \simeq k$. 
\end{lemma}

\begin{proof}
Since the natural map $\uk \lra k$ is a homology isomorphism, 
it suffices to construct a homology isomorphism $\uk \lra M$.
Writing $K_i =K(x_1,...., x_i ) $ as before, so that $K_0=R$ and $K_r =\uk$, 
there is a cofibre sequence
$$\Sigma^{|x_i|}K_{i-1} \lra K_{i-1} \lra K_i$$
for $i=1,2, \ldots, n$. Furthermore, since $x_1, \dots, x_r$ 
is a regular sequence, $H_*(K_i)=H^*(BG)/(x_1,\ldots , x_i)$.
We now successively construct maps $K_i\lra M$ for $i=0,1,\ldots r$
inducing epimorphisms
$$H^*(K_i)=H^*(BG)/(x_1,\ldots , x_i)\lra H_*(M)=k$$
in homology. Indeed, we choose a non-zero element of homology
to give $K_0=R \lra M$, and
we may extend $K_{i-1}\lra M$ over $K_i$, since $x_if_{i-1}$ is zero in 
homology and therefore nullhomotopic.
\end{proof}

Accordingly, by \ref{prop-cell-qe}, the Quillen equivalence 
$$\Rtmod\simeq \Ramod$$
induces a Quillen equivalence
$$\ktcellRtmodp\simeq \kkcellRamodp.$$
Combining this with the Quillen equivalences of Subsection \ref{subsec:piequiv}
and Subsection \ref{subsec:itorsequiv} completes the proof.
\end{proof}

\section{Change of groups.}
\label{sec:change}
In this section we consider restriction of equivariance and its
adjoints. Suppose then that $G$ is a compact Lie group and
that  $H$ is a closed  subgroup, with inclusion 
$$i:H\lra G.$$ 
The restriction functor
$$i^*=\res^G_H: \Gspec \lra \Hspec $$
has a left adjoint and a right adjoint. The left adjoint is induction
$$i_*=\ind_H^G(Y)=G_+ \sm_H Y$$
and the right adjoint is coinduction
$$i_!=\coind_H^G(Y)=F_H(G_+ ,  Y).$$
Furthermore, these are related by an equivalence
$$F_H(G_+,Y)\simeq G_+\sm_H(S^{-L(H)}\sm Y)$$
where $L(H)$ is the representation of $H$ on the tangent space to
$G/H$ at the identity coset. This suspension already signals the
greater sophistication of the right adjoint. 
When we restrict attention to free spectra and $H$ is connected, 
suspension by $L(H)$ is equivalent to a $G$-fixed suspension of the
same dimension. 

Now consider the situation in which we have given algebraic models,
where $G$ and $H$ are connected. We will write $r: H^*(BG) \lra
H^*(BH)$ for the map induced by $i$ to avoid conflict with the use of 
$i^*$ above. 

It can be seen from variance alone that there will be
some interesting phenomena involved in finding counterparts in the
algebraic model. We begin by motivating the answer with an outline
discussion. 

Note that at the strict level, restriction 
$$r^*:  \modcat{H^*(BH)} \lra \modcat{H^*(BG)} $$
has left and right adjoints
$$r_*,r_!: \modcat{H^*(BG)} \lra \modcat{H^*(BH)} $$
defined by extension and coextension of scalars
$$r_*(M)=H^*(BH)\tensor_{H^*(BG)} M \mbox{ and }
r_!(M)=\Hom_{H^*(BG)} (H^*(BH), M).$$
The functor corresponding to $\res^G_H$ for spectra has the variance
of  $r_*$ and $r_!$, and it turns out to correspond to $r_!$. Thus
{\em induction} $i_*=\ind_H^G$ (which is the leftmost of the adjoints
on spectra) corresponds to {\em restriction} $r^*$ (which is the
middle adjoint in the algebraic world). The counterpart in algebra of 
the  right adjoint $i_!$ is a functor we have not yet mentioned, 
and it is only easily apparent at the derived level.  The existence 
of this counterpart of the right adjoint $i_!$ at the derived level
depends critically on the fact that $H^*(BH)$ is a  finitely generated 
$H^*(BG)$-module by Venkov's theorem. Moving to derived categories, 
since $H^*(BG)$ is a polynomial ring, the finite generation means that
 $H^*(BH)$ is small. Accordingly, if we write 
$$DM=\Hom_{H^*(BG)}(M,H^*(BG))$$
for the derived dual of a module, smallness means that the natural map 
$$r_!(N)=\Hom_{H^*(BG)}(H^*(BH), N)\stackrel{\simeq}
\lra D(H^*(BH))\tensor_{H^*(BG)} N$$
is an equivalence in the derived category.  It is then apparent that
the right adjoint of the derived functor $r_!$ is the functor 
$$r^!(M)=\Hom_{H^*(BG)}(D(H^*(BH)), M). $$

\begin{thm}
If $G$ and $H$ are connected compact Lie groups, and the inclusion
$i:H\lra G$ induces $r: H^*(BG) \lra H^*(BH)$ then at the level of
homotopy categories
\begin{itemize}
\item  induction of spectra corresponds to restriction of scalars along
$r$. 
\item restriction of
free $G$-spectra corresponds to coextension of scalars along $r$ 
\item  coinduction of spectra corresponds to the functor $r^!$
\end{itemize}

More precisely, for the left adjoint of restriction, we have Quillen adjunctions
$$\adjunction{\ind_H^G}{\freeHspec}{\freeGspec}{\res^G_H}$$
and 
$$\adjunction{r^*}{\torsHBHmod}{\torsHBGmod}
{r_!}$$
where 
$$r_!(M)=\Hom_{H^*(BG)}(H^*(BH),  M). $$
The resulting derived adjunctions correspond under the Quillen equivalence of Theorem
\ref{thm:culmination} in the sense that for each step in the  zig-zag of Quillen
equivalences of the domain and codomain,  there are 
adjunctions which correspond at the derived level.

For the right adjoints,  we have Quillen
adjunctions $$\adjunction{\res_H^G}{\freeGspec}{\freeHspec}{\coind^G_H}$$
and 
$$\adjunction {r_!'}{\torsHBGmod}{\torsHBHmod}
{r^!}, $$
where 
$$r'_!(M) =DH^*(BH)\tensor_{H^*(BG)} M \simeq   \Hom_{H^*(BG)}(H^*(BH),  M )   =:r_!$$
 and
$$r^!(N) =\Hom_{H^*(BH)}(DH^*(BH),  N ).$$
The associated derived adjunctions of these pairs also correspond at the derived level. 
\end{thm}

\begin{proof}
We are in the situation of having a number of model categories $\C
(G)$ associated to $G$, and corresponding model categories $\C (H)$
associated to $H$.   Below we show that certain diagrams 
$$\diagram 
\C (G) \rto^{\simeq} &\D (G)\\
\C (H) \uto^{i_{\C}} \rto^{\simeq} &\D (H)\uto^{i_{\D}  }
\enddiagram$$
commute, where the horizontals are Quillen equivalences and the
verticals are left adjoints.  
The relevant categories $\C (G)$ and functors $i:\C(H)\lra \C(G)$
are 
\begin{enumerate}
\item  $\C_1(G)= \freeGspec$, $i(Y)=G_+\sm_H Y$
\item $\C_2(G)= \modcat{\Q[G]}$,  $i(Y)=\Q [G]\tensor_{\Q[H]}Y$
\item $\C_3(G)= \kcellRtopqmod$,  $i(N)=(r_{top}')^*(N)$
\item $\C_4(G)= \kcellRtoppmod$,  $i(N)=r_{top}^*(N)$
\item $\C_5(G)= \kcellRtmod$,  $i(N)=r_{t}^*(N)$
\item $\C_6(G)= \kcellHBGmodp$,  $i(N)=r_a^*(N)$
\item $\C_7(G)= \kcellHBGmodi$,  $i(N)=r_a^*(N)$ 
\item $\C_8(G)= \torsHBGmod$,  $i(N)=r_a^*(N)$
\end{enumerate}

The commutativity is clear for the diagrams involved in the comparisons
1-2, 3-4, 6-7 and 7-8.   The algebraicization
step 4-5 also follows since the equivalences of \cite{s-alg} are
natural. This leaves 2-3 and 5-6. 

For the rigidity step 5-6, we need to take care about the order of 
choices. Indeed, we choose $r_t=i^*_t:R_t(G)\lra R_t(H)$ to be a fibration 
of DGAs. Now consider the square
$$\diagram 
H^*(BG) \rdotted \dto &R_t(G)\dto\\ 
H^*(BH) \rto &R_t(H)
\enddiagram$$ 
We obtain the lower horizontal by choosing cycle representatives
$\yt_1, \ldots , \yt_s$ for the polynomial generators of 
$H^*(BH)=\Q [y_1,\ldots, y_s]$. Now if $H^*(BG)=\Q[x_1,
\ldots, x_r]$ and we choose cycle representatives $\xt_1', \ldots ,
\xt_r'$ for the polynomial generators, we will not obtain a commutative 
square. However if $i^*(x_s)=X_s(y)$, the classes
$i^*(\xt_s')$ and $X_s(\yt)$ are cohomologous, and therefore the   
differences $i^*(\xt_s')-X_s(\yt)$ are coboundaries. Since the map
$R_t(G)\lra R_t(H)$ is a fibration, we may lift this coboundary to a 
coboundary $d(e_i)$ in $R_t(H)$. Now take $\xt_s=\xt_s'-e_s$ and we 
complete the square to a commutative square of DGAs. 

Finally, the most interesting step is the Koszul step 2-3. The map
 $i: \Q [H]\lra \Q[G]$ induces the restriction map
$$r'_{top}=i^*_{top}: R'_{top}(G)=F_{\Q[G]}(k,k) \lra F_{\Q[H]}(k,k)
=R'_{top}(H) $$
where here $k$ is $\Q [EG]$, the $G$-cofibrant replacement of $\Q$ which is 
also an $H$-cofibrant replacement, since we may choose $EG$ as our model for
$EH$.  
We must show the diagram 
$$ 
\diagram 
\modcat{\Q[G]}\rto^{\hspace*{-4ex}k \otimes_{\Q[G]} (\cdot)}  & \kcellRtopqGmod\\
\modcat{\Q[H]}\rto^{\hspace*{-4ex}k \otimes_{\Q[H]} (\cdot)}\uto^{i_*}  &
\kcellRtopqHmod \uto_{r^*}
\enddiagram$$
commutes. Starting with a $\Q [H]$-module $Y$ we see that this amounts
to the isomoprhism
$$k \tensor_{\Q[G]} (Q[G] \tensor_{\Q[H]} Y )  \cong  k \tensor_{\Q[H]} Y.$$

The left adjoints for the categories (i), (ii), (vii) and (viii) are left Quillen functors and hence 
induce associated derived functors.   For the standard model structures
in (iii) through (vi) the functors $r^*$ preserve all weak equivalences, so they also induce
derived functors.      This shows that the derived functors of the left adjoints for the categories
in (i) and in (viii) correspond.   Thus their derived right adjoints also correspond.

Now consider the second Quillen pairs listed in the statement of
the theorem.   The adjunction described at the level of spectra 
(i.e., $\C_1$) is a Quillen pair by~\cite[V.2.2]{mm}.  
In algebra (i.e., $\C_{8}$), $r'_{!}$ preserves homology isomorphisms
and injections, hence it is a left Quillen functor.
As shown above the statement of the theorem, at the derived level
$r_!$ is naturally isomorphic to $r_!'$. 
Since the first statement in the theorem shows that at the derived level $\res^{G}_{H}$
corresponds to $r_{!}$, it follows that the second derived adjunctions also correspond.
\end{proof}

\begin{appendix}
\section{Cellularization of Model Categories}\label{sec-cell}

Throughout the paper we need to consider models for categories of
cellular objects, thought of as constructed from a basic cell using
coproducts and cofibre sequences. These models are usually obtained
by the process of cellularization (or colocalization) of model categories, with the
cellular objects appearing as the cofibrant objects. Because it is is 
fundamental to our work, we recall some of the basic definitions from~\cite{hh}.

\begin{defn}\cite[3.1.8]{hh}
Let $\M$ be a model category and $A$ be an object in $\M$.
A map $f: X \to Y$ is an {\em $A$-cellular equivalence}
if the induced map of homotopy function complexes~\cite[17.4.2]{hh} 
$f_*: \map(A, X) \to \map(A, Y)$ is a weak equivalence. 
An object $W$ is {\em $A$-cellular} if $W$ is cofibrant in $\M$ and 
$f_*: \map(W, X) \to \map(W, Y)$ is a weak equivalence for
any $A$-cellular equivalence $f$.
\end{defn}

\begin{prop}\cite[5.1.1]{hh}\label{prop 5.5}
Let $\M$ be a right proper  model category which is cellular in the
sense of \cite[12.1.1]{hh} 
and let $A$ be an object in $\M$.  The $A$-cellularized
model category $\AcellM$ exists and has
weak equivalences the $A$-cellular equivalences,
fibrations the fibrations in $\M$ and cofibrations
the maps with the left lifting property with respect
to the trivial fibrations. 
The cofibrant objects are the $A$-cellular objects.
\end{prop}

It is useful to have the following further characterization of the
cofibrant objects.   

\begin{prop}\cite[5.1.5]{hh}\label{prop-cell-obj}
If $A$ is cofibrant in $\M$, then the class of $A$-cellular objects 
agrees with the smallest class of cofibrant ojects in $\M$ that
contains $A$ and is closed under homotopy colimits and weak
equivalences.  
\end{prop}

Since we are always working with stable
model categories here, homotopy classes of maps out of $A$ detects
trivial objects.   That is, in $\Ho(\AcellM)$, an object $X$
is trivial if and only if $[A, X]_{*}=0.$  In this case,
by \cite[2.2.1]{ss2} we have the following.

\begin{prop}\label{prop-gen}
If $\M$ is stable and $A$ is compact, then 
$A$ is a generator of $\Ho(\AcellM)$.   That is,
the only localizing subcategory containing $A$
is $\Ho(\AcellM)$ itself.  
\end{prop}

We next need to show that appropriate cellularizations of these
model categories preserve Quillen equivalences.

\begin{prop}\label{prop-cell-qe}
Let $\M$ and $\N$ be right proper cellular model
categories with $F: \M \to \N$ a Quillen 
equivalence with right adjoint $U$.  
\begin{enumerate}
\item Let $A$ be an object in $\M$, then $F$ and $U$
induce a Quillen equivalence between the $A$-cellularization
of $\M$ and the $FQA$-cellularization of $\N$ where
$Q$ is a cofibrant replacement functor in $\M$. 
\[ \AcellM \simeq_{Q} \mbox{$FQA$-cell-$\N$}\]
\item Let $B$ be an object in $\N$, then $F$ and $U$
induce a Quillen equivalence between the $B$-cellularization
of $\N$ and the $URB$-cellularization of $\M$ where
$R$ is a fibrant replacement functor in $\N$. 
\[ \mbox{$URB$-cell-$\M$} \simeq_{Q} \mbox{$B$-cell-$\N$}\]
\end{enumerate} 
\end{prop}

In~\cite[2.3]{hovey-stable} Hovey gives criteria for when
localizations preserve Quillen equivalences.  Since
cellularization is dual to localization, this proposition
follows from the dual of Hovey's statement.   

\begin{proof}
First note that $FQURB \to B$ is a weak equivalence
since $F$ and $U$ form a Quillen equivalence.
The criterion in~\cite[3.3.18(2)]{hh} (see also~\cite[2.2]{hovey-stable}) 
for showing that $F$ and $U$ 
induce a Quillen adjoint pair on the cellular model categories
follows in (i) automatically and in (ii) since $FQURB$ is weakly equivalent to $B$.
Similarly, in each case the choice of the cellularization objects establishes 
the criterion for Quillen equivalences in~\cite[2.3]{hovey-stable}.  

In more detail, dual to the proof of~\cite[2.3]{hovey-stable},
we note that for every cofibrant object $X$ in 
$\N$, the map $X \to URFX$ is a weak equivalence. Since
fibrant replacement does not change upon cellularization
and cellular cofibrant objects are also cofibrant in $\N$, 
we apply~\cite[1.3.16]{hovey-model} to the Quillen adjunctions
of the cellular model categories.  Thus,
in each case we only need to show that $U$ reflects weak equivalences
between fibrant objects. 
In case (i), we are given $f: X \to Y$ a weak
equivalence between fibrant objects such that $Uf$ is an $A$-cellular
equivalence in $\M$.  Thus $Uf_*: \map(A, UX) \to \map(A, UY)$       
is an equivalence.  By~\cite[17.4.15]{hh}, it follows that
$f_*: \map(FQA, X) \to \map(FQA, Y)$ is also a weak equivalence.
Thus $f$ is a $FQA$-cellular equivalence as required.       
In case (ii), we are given $f: X \to Y$ a weak
equivalence between fibrant objects such that $Uf$ is an $URB$-cellular
equivalence in $\M$.  Thus $Uf_*: \map(URB, UX) \to \map(URB, UY)$       
is an equivalence.  By~\cite[17.4.15]{hh}, it follows that
$f_*: \map(FQURB, X) \to \map(FQURB, Y)$ is also a weak equivalence.
Since $FQURB \to RB$ is a weak equivalence, we see that $f_*$
is a $B$-cellular equivalence as required.  
\end{proof}

\end{appendix}

\end{document}